\documentclass{amsart}
\usepackage{hyperref}

\usepackage{amsmath,  amssymb, amscd, graphicx}
\usepackage{mathtools}
\usepackage{comment}
\usepackage[all]{xy}
\usepackage{fullpage,enumerate}
\newcommand\ad{\operatorname{ad}}
\newcommand\vol{\operatorname{vol}}

\newcommand\oM{\overline{M}}
\newcommand\ff{\operatorname{ff}}
\newcommand\scl{\operatorname{sl}}

\newcommand\bW{\overline{W}}
\numberwithin{equation}{section}

\newcommand\paperbody%
        {}

\newcommand\rbm[1]{\mar{RBM:#1}}
\newcommand\rbmoff{\renewcommand\rbm[1]{}}
\newcommand{\mar}[1]{{\marginpar{\sffamily{\scriptsize #1}}}}
\theoremstyle{plain}

\newtheorem{thm}{Theorem}
\newtheorem{theorem}{Theorem}[section]
\newtheorem*{thm*}{Theorem}

\newtheorem*{cor*}{Corollary}

\newtheorem*{conj*}{Conjecture}
\newtheorem*{lemma*}{Lemma}

\newtheorem{lemma}[thm]{Lemma}
\newtheorem*{prop*}{Proposition}

\newtheorem{proposition}[thm]{Proposition}
\theoremstyle{definition}
\newtheorem{rems}[thm]{Remarks}

\newtheorem*{defn*}{Definition}

\newtheorem*{rems*}{Remarks}
\newtheorem*{proof*}{Proof}
\newtheorem{prel*}{Preliminaries}
\newtheorem{examples*}{Examples}
\newcommand\ha{\frac12}
\newcommand\ev{\operatorname{ev}}

\newcommand\TT{\mathbb{T}}

\newcommand{\End}{{\rm{End}}}

\newcommand\coF{{}^{\mathcal{C}}\kern-2pt\Lambda}

\newcommand\cFTs{{}^{\Phi}\overline{T}\kern-1pt{}^*}

\newcommand\even{\text{even}}
\newcommand\odd{\text{odd}}

\newcommand\fib{\operatorname{fib}}

\newcommand\Tr{\operatorname{Tr}}

\newcommand\Cl{\operatorname{Cl}}

\newcommand\SpinC{\operatorname{Spin}_{\CC}}

\newcommand\GL{\operatorname{G}}

\newcommand\Ch{\operatorname{Ch}}
\newcommand\Td{\operatorname{Td}}
\newcommand\com[1]{\overline{#1}}

\newcommand\cA{\mathcal{A}}

\newcommand\cF{\mathcal{F}}

\newcommand\cG{\mathcal{G}}
\newcommand\dcG{\dot{\mathcal{G}}}

\newcommand\cL{\mathcal{L}}

\newcommand{\cP}{{\mathcal P}}
\newcommand{\dcP}{\dot{\mathcal P}}

\newcommand\CC{\mathbb C}

\newcommand\ZZ{\mathbb Z}
\newcommand\bbH{\mathbb H}

\newcommand\bbC{\mathbb C}

\newcommand\bbR{\mathbb R}
\newcommand\bbS{\mathbb S}

\newcommand\bbZ{\mathbb Z}

\newcommand\cM{\mathcal M}
\newcommand\cS{\mathcal S}
\newcommand\CIc{{\mathcal{C}}^{\infty}_c}

\newcommand\CI{{\mathcal{C}}^{\infty}}

\newcommand\Diag{\operatorname{Diag}}

\newcommand\cFNs{{}^{\Phi}\overline N\kern-1pt{}^*}

\newcommand\ind{\operatorname{ind}}

\newcommand\tr{\operatorname{tr}}

\newcommand\Id{\operatorname{Id}}

\newcommand\PU{\operatorname{PU}}

\newcommand\Ran{\operatorname{Ran}}

\newcommand\ci{${\mathcal{C}}^\infty$}

\newcommand\dCI{\dot{\mathcal{C}}^{\infty}}

\newcommand\pa{\partial}

\newcommand\inn{\operatorname{int}}

\newcommand\dR{\operatorname{dR}}

\newcommand\Mand{\text{ and }}
\newcommand\Mas{\text{ as }}
\newcommand\Mat{\text{ at }}

\newcommand\Mie{\text{ i.e. }}

\newcommand\Mst{\text{ s.t. }}

\newcommand\tb{\operatorname{tb}}

\rbmoff
\begin{document}

\title[Small gerbes]{Index and small bundle gerbes}

\author{Varghese Mathai}
\address{Department of Pure Mathematics, University of Adelaide,
Adelaide 5005, Australia}
\email{mathai.varghese@adelaide.edu.au}
\author{Richard B. Melrose}
\address{Department of Mathematics,
Massachusetts Institute of Technology,
Cambridge, Mass 02139, U.S.A.}
\email{rbm@math.mit.edu}

\dedicatory{Dedicated to the memory of Isadore M. Singer}

\begin{abstract} By a small bundle gerbe we mean a bundle gerbe in the
  sense of Murray defined on a smooth, finite-dimensional, fibre bundle
  over a manifold. We construct such gerbes over compact oriented
  aspherical 3-manifolds, as well as in higher dimensions, generalizing the
  construction of decomposable bundle gerbes in earlier work with
  Singer. For these small bundle gerbes there is a direct index map given in
  terms of either fibrewise pseudodifferential operators, or more
  conveniently fibrewise semiclassical smoothing operators, twisted by the
  simplicial line bundle. We prove the Atiyah-Singer type theorem that this
  realizes the push-forward into twisted K-theory. We also give an
  application via the index of projective families of
  $\operatorname{Spin}_{\bbC}$ Dirac operators, to show the existence of
  obstructions to metrics with large positive scalar curvature.
\end{abstract}

\thanks{{\em Acknowledgments.} The first author was supported by the
  Australian Laureate Fellowship FL170100020.}
  
\maketitle

\section*{Introduction}

Bundle gerbes give a geometric realization of integral 3-classes. Such a
class on a space $M$ is represented as the transgression, to the
Dixmier-Douady class, of the Chern class of a simplicial line bundle for a
space with a split surjective map $\phi:F\longrightarrow  {Y}.$ In fact each
3-class can be represented by a bundle gerbe with $F$ an
infinite-dimensional bundle, for instance a path space. We consider below
\emph{small gerbes,} corresponding to the special case where $F$ is a
smooth, compact, fibre bundle over $M.$ As shown by Murray and Stevenson
\cite{MS3} if $M$ is simply connected then only torsion classes arise this
way. Conversely, in \cite{MMS1,MMS2}, it was shown that decomposed
3-classes can be represented by such small gerbes. In \cite{Bigerbes} it is
noted that a 3-class on $M$ is represented by a bundle gerbe over a fibre
bundle $F$ if and only if the 3-class pulls back to be trivial over $F.$ It
follows that if a 3-class on $ {Y}$ lifts to a fibre bundle over $M$ to be
torsion or decomposable then it is represented by a small gerbe. In this
way we introduce small gerbes for 3-classes which are neither torsion nor
decomposable. Although we leave open the question of precisely which
integral 3-classes on a compact manifold are represented by small gerbes we
show for instance that any compact oriented aspherical 3-manifold supports
a small gerbe representing its volume class.

 {If $\xymatrix{Z\ar@{-}[r]&F\ar[r]^{\phi}&M}$ is a smooth fibre bundle,
with compact typical fibre $Z,$ the fibre products
$F^{[k]}=\{(p_1,\dots,p_k)\in F^k;\phi(p_1)=\dots=\phi(p_k)\}$ form a
  simplicial space. The degeneracy maps are the projections off factors and
  the face maps are the diagonal embeddings between two adjacent
  factors. In the case of a gerbe we are mostly interested in the spaces
  with projections $\pi_1,\pi_2:F^{[2]}\longrightarrow F$ (off the right and
left factor respectively) and $\pi_{12},\pi_{13},\pi_{23}:F^{[3]}\longrightarrow
F^{[2]}$ off the right-most, central and left-most factors. A complex line
bundle $J\longrightarrow F^{[2]}$ is simplicial if the combined pull back
is trivial
$s:\pi_{12}^*J\otimes \pi_{13}^*J^{-1}\otimes \pi_{23}^*J\simeq\bbC$ with the
trivialization reducing to the natural trivialization over $F^{[4]}.$
In this finite dimensional case a gerbe defines a twisted version of the
fibre-wise pseudodifferential operators on $F.$} The subspace of smoothing
operators, $\Psi^{-\infty}_{\phi,J}(F)=\cA_J,$ is defined by sections
$A\in\CI(F^{[2]};J\otimes \pi_R^*\Omega ),$ where $\Omega $ is the
(trivial) bundle of fibre densities on $F,$ with composition
\begin{equation}
A\circ B(z,z')=\int_{Z}A(z,\cdot)\circ B(\cdot,z')
=(\pi_{13})_*\left(\pi_{12}^*A\circ \pi_{23}^*B\right).
\label{IBG.9}\end{equation}
 {Here the composite in the integrand includes the simplicial product on
$J,$ $\pi_{12}^*J\otimes \pi_{23}^*J\simeq\pi_{13}^*J.$} We call such a bundle over
$Y,$ arising from a small gerbe, a \emph{smooth 
Azumaya bundle.} It has a completion, $\overline{\cA_J},$ to an Azumaya
bundle, modelled on the compact operators on a Hilbert space, with
Dixmier-Douady class the 3-class represented by the gerbe. Defining  
\begin{equation}
\cG^{-\infty}_{\phi,J}(F)=\{a\in \Psi^{-\infty}_{\phi,J}(F);\Id+a\text{ is invertible}\}
\label{IBG.178}\end{equation}
gives a classifying bundle for the odd twisted K-theory of the base 
\begin{equation}
K^1(M;\overline{\cA_J})=\cG^{-\infty}_{J,\psi}(F)/\text{smooth homotopy.}
\label{IBG.179}\end{equation}

To define and manipulate the twisted index we use a semiclassical version
of the algebra \eqref{IBG.9}. On a fixed compact manifold, $Z,$ the
semiclassical smoothing operators are smooth families of smoothing
operators $A_\epsilon \in\CI((0,1]_\epsilon;\Psi^{-\infty}(Z))$ depending
  on a parameter $\epsilon \in(0,1]$ with distinct singular behaviour at
    $\epsilon =0.$ The kernel of $A_\epsilon\in\CI((0,1]\times
      Z^2;\pi_R^*\Omega )$ is required to vanish rapidly with all
      derivatives as $\epsilon \downarrow0$ away from the diagonal of
      $Z^2;$ as usual $\Omega$ is the density bundle on $Z.$ Near the
      diagonal, in local coordinates, the kernel should take the form
\begin{equation*}
A_\epsilon (z,z')=\epsilon ^{-d}\alpha (\epsilon ,z,\frac{z-z'}\epsilon
)|dz'|,\ d=\dim Z,
\label{IBG.180}\end{equation*}
where $\alpha$ is smooth down to $\epsilon =0$ and is a Schwartz function
in the last variables. These conditions are coordinate-invariant and so
extend naturally to fibre bundles $\phi:F\longrightarrow Y.$ The behaviour
of the kernels is conveniently captured in terms of a blown-up version of the fibre product 
\begin{equation}
F[2,\scl]=[F^{[2]}\times[0,1];\Diag\times\{0\}]\overset{\beta}
\longrightarrow F^{[2]}\times[0,1]
\label{IBG.181}\end{equation}
where they become 
\begin{equation}
\Psi^{-\infty}_{\phi,\scl}(F;W)=\{A\in\epsilon
^{-d}\CI(F[2,\scl];\pi_L^*W\otimes\pi_R^*W'\otimes\pi_R^*\Omega);A\equiv0\Mat H_0\}.
\label{IBG.182}\end{equation}
Here $H_0$ is the lift, or proper transform, of the boundary face
$\epsilon =0$ and $W$ is a complex vector bundle over $F.$ These
form an algebra of operators acting on $\CI(F\times[0,1];W)$ with two significant
ideals 
\begin{equation}
\epsilon ^{\infty}\Psi^{-\infty}_{\phi,\scl}(F;W)\subset \epsilon\Psi^{-\infty}_{\phi,\scl}(F;W).
\label{IBG.183}\end{equation}
The first of these is naturally identified with the families of
smoothing operators $\CI([0,1];\Psi^{-\infty}_{\phi}(F;W))$ which vanish to
infinite order at $\epsilon =0.$ The second ideal corresonds to simple
vanishing at $H_0$ and is captured by the semiclassical symbol giving a
short exact sequence
\begin{equation}\xymatrix{
\epsilon\Psi^{-\infty}_{\phi,\scl}(F;W)\ar[r]&
\Psi^{-\infty}_{\phi,\scl}(F;W)\ar[r]^-{\sigma _{\scl}}&
\dCI(\overline{T^*_\phi F};\hom W).
}
\label{IBG.184}\end{equation}
The image is the space of Schwartz sections on the fibrewise
cotangent bundle $T ^*_\phi F,$ realized as the smooth functions on the
radial compactification vanishing to infinite order at the boundary (`at infinity').
The quotient by the smaller ideal is a $*$ algebra, with product
essentially that of the full symbol algebra of pseudodifferential
operators.

The relevance of these operators for index theory is that there is a
natural map
\begin{equation}
\dcG(\overline{T^*_\phi F};\hom W)=\{a\in \dCI(\overline{T^*_\phi F};\hom
W);\Id+a\text{ is invertible}\}\longrightarrow K^1(T^*_{\psi}F)
\label{IBG.185}\end{equation}
which captures all K-classes when stabilized over $W.$

The (untwisted) semiclassical index map arises by considering the
corresponding group of invertibles 
\begin{equation*}
\cG^{-\infty}_{\phi,\scl}(F;W)=\{A\in
\Psi^{-\infty}_{\phi,\scl}(F;W);\Id+A\text{ invertible}\}
\label{IBG.186}\end{equation*}
for which the symbol map in \eqref{IBG.184} becomes an homotopy equivalence 
\begin{equation}
\cG^{-\infty}_{\phi,\scl}(F;W)\overset{\sigma_{\scl}}\longrightarrow \dcG(\overline{T^*_\phi F};\hom W)
\label{IBG.187}\end{equation}
and the index is given by the equivalence class of the restriction map 
\begin{equation*}
\ind_{\scl}:\cG^{-\infty}_{\phi,\scl}(F;W)\ni A\longmapsto A_{\epsilon =1}.
\label{IBG.190}\end{equation*}

The index map in K-theory is given by stabilization, say over $W=\bbC^N$ for all
$N:$  
\begin{equation}
\xymatrix{
&\dcG(\overline{T^*_{\phi}F};\bbC^N)\ar[r]& K^1(T^*_{\fib}F)\ar@{-->}[dd]^{\ind_{\scl}}\\
\cG^{-\infty}_{\phi,\scl}(F;\bbC^N)\ar@{->>}[ur]^{\sigma _{\scl}}\ar[dr]_{\epsilon =1}\\
&\cG^{-\infty}_{\phi}(F;\bbC^N)\ar[r]&K^1(Y).
}
\label{IBG.188}\end{equation}
We recall below, and elaborate on, the proof from \cite{MMS2} of the
Atiyah-Singer type theorem that this map is the push-forward in K-theory
and so is equivalent to pseudodifferential quantization; this can
be seen directly.

To construct the twisted index map we replace the smoothing operators on
the fibres of $F$ by the twisted smoothing operators in \eqref{IBG.9}. The
symbols of these operators again correspond to the leading singularity of the
kernel at the diagonal of $F^{[2]}.$ Since $J$ is trivial there they
take values in the same, untwisted space as in \eqref{IBG.184} although
this identification does depend on the simplicial product on $J.$
The restriction map, from the analogous group of semiclassical smoothing
operators, at $\epsilon =1$ takes values in the group \eqref{IBG.178} so
resulting in the twisted, semiclassical, map
\begin{equation}
\ind_{J,\scl}:K^1(T^*_{\fib}F)\longrightarrow K^1(M;\cA_J).
\label{IBG.11}\end{equation}
There is a similar construction, using idempotents, in the even case and
then the analogue of the (families) Atiyah-Singer theorem is:

\begin{thm}\label{IBG.13} The index map \eqref{IBG.11} is equal to
  the topological push-forward map.
\end{thm}

 To prove this we follow the, essentially standard, proof passing through
 the Gysin map using an embedding of $F$ in a trivial bundle $Y\times
 \bbR^N.$ Note however that the small gerbe does not extend from $F$ to
 define a gerbe over the trivial bundle. So we proceed in two steps by
 considering the index map for a general Azumaya bundle, $\cA_J,$
 defined by a small gerbe over $M$ pulled back to any smooth fibre bundle
 $\tilde F\longrightarrow M.$ Now semiclassical (or pseudodifferential)
 quantization leads to an index map
\begin{equation}
  \ind_{\cA_J}:K^*(T^*_{\fib}\tilde F;\phi^*\cA_J)\longrightarrow K^*(M;\cA_J).
\label{IBG.12}\end{equation}
In the case $\tilde F=F,$ corresponding to the small gerbe itself, this
reduces to \eqref{IBG.11}, since $\phi^*\cA_J$ is trivial as an Azumaya
bundle. In general for \eqref{IBG.12} we show that the index map factors
through any embedding and that this gives the standard push-forward map in
twisted K-theory.

In the first section the notion of a bundle gerbe is recalled. The geometry
of 3-manifolds is discussed in Section 2. The behaviour of bundle gerbes
for directed sums is described in Section 3 and the smooth Azumaya bundles
are introduced in Section 4. In Sections 5 and 6 some properties of
semiclassical quantization are given. The semiclassical version of the
families Atiyah-Singer theorem is stated in Section 7 and proved in Section
8. The twisted version follows by close analogy in Section 9 and the
range-twisted version is in Section 10. The Chern character is extended to
appropriate infinite-rank bundles in Section 11 and the odd and twisted
versions follow in Section 12 and 13. In Section 14 we give an application
of the index of projective families of $\SpinC$ Dirac operators, showing an
obstruction to the existence of large positive scalar curvature.

In \cite{BG}, the authors consider a fibre bundle $\pi: M\longrightarrow B$ of compact manifolds and a gerbe 
$L$ on $B$ such that $\pi^*(L)$ is a gerbe on $M$ having torsion Dixmier-Douady class. 
They define projective families of pseudodifferential operators 
and in particular, projective families of generalised Dirac operators. They use the results of \cite{MS} to see that the
Chern character of the analytic index lands in the twisted cohomology ${H^*}_L(B)$. Finally, they
generalise Bismut's heat kernel approach \cite{Bismut} to the index theorem for families of generalised Dirac operators,
to calculate the local index for projective families of generalised  Dirac operators, and obtain their main result, 
which is a special case of our Theorem \ref{IBG.191}.  In particular, their approach to the 
index theorem is interesting in its own right, but it does ignore torsion in K-theory, whereas our main
result is the index theorem in 
K-theory for projective families of elliptic operators, Theorem \ref{IBG.13}. We also give new constructions 
of the gerbes $L$ on $B$ in \S \ref{sect:Thurston}

%

\paperbody 

\section{Small bundle gerbes}

By a \emph{small gerbe} we will mean a smooth, connected, compact, finite-dimensional bundle
gerbe in the sense of Murray. More precisely the data is a triple
$(M,F,J)$ where $\phi:F\longrightarrow M$ is a fibration of compact
connected manifolds and $J$ is a simplicial line bundle over $F^{[2]}.$
Equivalently because of compactness, $\phi$ is a submersion. The simplicial
condition on the line bundle over the fibre product 
\begin{equation}
F^{[2]}=\{(p,p')\in F^2;\phi(p)=\phi(p')\}
\label{IBG.26}\end{equation}
is the existence of a smooth bundle isomorphism  
\begin{equation}
\sigma :\pi_{12}^*J\otimes\pi_{23}^*J\longrightarrow \pi_{13}^*J\text{ over }F^{[3]}
\label{IBG.27}\end{equation}
inducing the canonical trivialization over $F^{[4]}.$ This implies the
local triviality of the gerbe, meaning the existence over the preimages in
$F$ of some open cover of $M$ of line bundles $K$ such that $J$ is
isomorphic, with its simplicial product, to $\pi_1^*K\otimes\pi_2^*K^{-1}.$

\begin{rems} Given a good open cover $\{U_\alpha\}$ of $M$, one can always
  consider the disjoint union $\amalg\, U_\alpha$ and projection map
  $p:\amalg\, U_\alpha \rightarrow M$ which is a submersion.  This is known
  as the {\em \v Cech bundle gerbe}.  Although this is finite dimensional,
  the total space is highly disconnected, making it unsuitable for index
  theory applications.
\end{rems}

\begin{proposition}[\cite{Bigerbes}]\label{bg criterion} Let
  $\pi:F\longrightarrow M$ be a finite-dimensional compact fibre bundle
  then $F$ carries a small bundle gerbe with Dixmier-Douady 3-class $\delta
  \in H^3(M;\bbZ)$ if and only if $\pi ^*\delta =0,$ \end{proposition}

As analysed in \cite{MMS2}, decomposable 3-classes give rise to small
bundle gerbes. More generally we note that

\begin{theorem}\label{thm:almost decomposable} If $H$ is a degree 3
  integral cohomology class on a manifold $N$ with a finite cover $\hat
  N\longrightarrow N,$ on which the lift of $H$ is decomposable, then $N$
  carries a small bundle gerbe with Dixmier-Douady invariant equal to $k
  H$, for any $k\in \ZZ$.  \end{theorem}

\begin{proof} The decomposable case is treated in \cite{MMS2}, so there
  exists a smooth compact fibre bundle over $\hat{M}$ to which $H$ lifts to
  be trivial. This defines a smooth compact fibre bundle over $M$ with the
  same property, so Proposition~\ref{bg criterion} applies.
\end{proof}

\section{Thurston's list of geometric 3-manifolds}\label{sect:Thurston}

There are many examples of compact oriented aspherical $3$-manifolds. 
In Thurston's list of eight geometries, six of them are aspherical:

\begin{enumerate}
\item hyperbolic 3-manifolds i.e.\ quotients of hyperbolic space $\bbH^3,$ 
\item flat manifolds i.e.\ quotients of Euclidean space $\bbR^3$, 
\item nilmanifolds i.e.\ quotients 
of the Heisenberg group ${\rm Heis}$ which is the $\bbR$ central extension of $\bbR^2$,
\item solvmanifolds i.e.\ quotients
of  the solvable group ${\rm Solv}$  which is an $\bbR^2$ extension of $\bbR$, 
\item quotients
of $\bbH^2 \times \bbR$ 
\item quotients
of  $\widetilde{\operatorname{SL}(2, \bbR)}$ which is the universal cover
of $\operatorname{SL}(2, \bbR).$
\end{enumerate}

The discussion in Otal \cite[page 15, following problem 17]{Otal14}
asserts that

\begin{theorem} \label{thm:Otal}
Any compact oriented aspherical $3$-manifold $M$ has a finite cover $\hat M$
with positive first Betti number $b^1(\hat M)>0$.
\end{theorem}

To prove this Otal uses a result of Agol \cite{Agol} as well as a result of
Haglund and Wise \cite{HW}.

\begin{proposition}\label{prop:asph}
Any compact oriented aspherical $3$-manifold $M$ has a small
bundle gerbe with Dixmier-Douady invariant equal to a multiple of the
volume class $k\operatorname{vol}(M)$, for any $k\in \ZZ$.
\end{proposition}

\begin{proof} Let $\hat M$ be a finite cover as in Theorem \ref{thm:Otal}.
  It follows that $\hat M$ has a decomposable volume form. Let $a\in
  H^1(\hat M;\bbZ)$ be a class with non-zero rational image.  Since $\hat
  M$ is compact and orientable, by Poincar\'e duality there is class $b\in
  H^2(\hat M;\bbZ)$ such that $a\cup b$ is non-vanishing in $H^3(\hat M;\bbZ)=\bbZ.$

Since $\vol(M)$ pulls back to an integral multiple of $\vol(\hat M),$ a
multiple of the volume form on $M$ pulls back to be decomposable on $\hat
M.$ Now Theorem~\ref{thm:almost decomposable} shows that this class on
$\hat M$ is represented by a small bundle gerbe. Thus $k\vol(M)$ pulls back
to be torsion on the iterated fibre bundle over $M.$ Thus there is a small
bundle gerbe with Dixmier-Douady class $k\vol(M).$
\end{proof}
 
Compact oriented hyperbolic 3-manifolds are examples of compact oriented
aspherical 3-manifolds. There are many examples of closed hyperbolic
3-manifolds with vanishing first Betti number.  For example, most Dehn
fillings on most knot complements will have this property.  More precisely,
if $K$ is any knot in $S^3$ which is not a torus knot or satellite knot,
then $M=S^3-K$ has a complete hyperbolic structure of finite volume by the
work of Thurston.  The $(p,q)$ Dehn filling on $M$ (i.e. attaching a solid
torus so $p$(meridian) $+q$(longitude) bounds a disk, $p,q$ being co-prime
integers) gives a closed manifold with $H_1 = \ZZ/p\ZZ$ and this will be
hyperbolic for all but a finite number of choices of $(p,q)$ by Thurston's
hyperbolic Dehn filling theorem.  In fact, for many choices of $K$ and
almost all choices of $n,(1,n)$ Dehn filling will give a hyperbolic
homology $3$-sphere.  Therefore it is necessary in these cases to consider
a finite cover that has positive first Betti number.

Compact flat oriented 3-manifolds are also examples of compact oriented
aspherical 3-manifolds. An example is the manifold that has zero first
Betti number, see \cite{Wolff}. It is necessary in this case also to
consider a finite cover that has positive first Betti number.

Note that $S^2\times S^1,$ the seventh example in Thurston's list of
$8$-geometries, gives an example of a $3$-manifold which is not aspherical,
but does have a small bundle gerbe, since the volume form on $S^2\times
S^1$ is decomposable.

The last example in Thurston's list of geometries is $S^3.$ Let $G$ be
a finite group acting freely on $S^3.$ Then

\begin{lemma}
$S^3/G$ does not have a nontrivial small bundle gerbe. 
\end{lemma}

\begin{proof} Suppose $(\cL, X)$ is a small bundle gerbe over $S^3/G$ with
  $p:S^3 \to S^3/G$ the projection.  Then $p^*(\cL, X)$ is a small bundle
  gerbe over $S^3$. But the main theorem in \cite{MS3} implies that there
  are no finite dimensional bundle gerbes over $S^3$, which is a
  contradiction. The Dixmier-Douady class of the gerbe therefore pulls back
  to be trivial on $\bbS^3$ but since this is a finite cover this class
  vanishes on the quotient.
\end{proof}

\section{Connected sums}

Let $M_i,$ $i=1,2$ be two smooth connected real $n$-dimensional manifolds.
Let $\iota:D_1\hookrightarrow M_1$ and $j:D_2\hookrightarrow M_2$ be embedded
balls. Consider their connected sum given by
$$
M_1\#M_2 = (M_1\setminus
\mathring{D}_1) \bigcup_\phi (M_2\setminus \mathring{D}_2)
$$
where $\phi: \partial D_1 \to \partial D_2$ is an orientation reversing
diffeomorphism of the $(n-1)$-sphere. Let $(\cL_i, X_i)$ be small bundle
gerbes over the $M_i.$  {Over} the $D_i$ the gerbes can be trivialized. Then
each of the bundles and gerbes can be extended as a small bundle gerbe to
the connected sum $M_1\#M_2 $, as the given bundle gerbe on $M_i\setminus
D_i$ and the trivial bundle gerbe on the other manifold. Denote these
extended small bundle gerbes as $(\widetilde \cL_i, \widetilde X_i)$.
Finally, consider the tensor product $(\widetilde \cL_1, \widetilde X_1)
\otimes (\widetilde\cL_2, \widetilde X_2):= (\cL_1, X_1) \# (\cL_2, X_2)$,
which is a small bundle gerbe over the connected sum $M_1\#M_2 $.  By induction, this
proves

\begin{lemma}\label{lem:connected-sum}
Let $M_i,$ $i=1,2,\ldots, m$ be  smooth connected $n$-dimensional manifolds having 
small bundle gerbes $(\cL_i, X_i)$ on $M_i$, $i=1, \ldots, m$. Then the connected sum $M=M_1\#M_2\#\ldots\# M_m$ has a small
bundle gerbe $(\cL_1, X_1) \# (\cL_2, X_2)\#\ldots \#(\cL_m, X_m)$ as above.
\end{lemma}

Then by Proposition \ref{prop:asph} and Lemma \ref{lem:connected-sum} we deduce our main result for 3-manifolds.

\begin{theorem}\label{thm:small gerbes on 3-manifolds}
Let $M_i,$ $i=1,2,\ldots, m$ be  smooth connected aspherical $3$-dimensional manifolds having 
small bundle gerbes $(\cL_i, X_i)$ on $M_i$, $i=1, \ldots, m$. Then the connected sum $M=M_1\#M_2\#\ldots\# M_m$ has a small
bundle gerbe $(\cL_1, X_1) \# (\cL_2, X_2)\#\ldots \#(\cL_m, X_m)$ as above.
\end{theorem}

We propose the following problem.\\

{\em Given a closed 3-manifold $M = M_1 \bigcup_{\TT^2} M_2$ which is a union of $M_1, M_2$ over an incompressible torus
$\TT^2$, 
and  small bundle gerbe $X_1 \to M_1$ with Dixmier-Douady invariant $H_1$, is there an extension to a 
small bundle gerbe $X\to M$?}

\section{Small Azumaya bundles}\label{S.SAB}

As noted in the Introduction, if $F$ is the total space of a small gerbe then
\begin{equation}
\Psi^{-\infty}_J=\CI(F^{[2]};J\otimes \Omega ),
\label{IBG.28}\end{equation}
is a bundle of algebras over the base $Y$ of the fibration. A local
trivialization of $F$ as a gerbe over $Y$ induces a local trivialization of
$\Psi^{-\infty}_J.$ Namely over the preimage $(\phi^{[2]})^{-1}(U_a)\subset F^{[2]}$
of the open cover $U_a\subset Y$ there is a smooth simplicial isomorphism $J\equiv
S_a\otimes S_a^{-1}$ for a smooth line bundle $S_a$ over
$F_a=\phi^{-1}(U_a)\subset F.$ This identifies $\Psi^{-\infty}_J(U_a)$ with the bundle of
smoothing algebras on the fibres $\Psi^{-\infty}_{\phi}(F_a;S_a).$ Over the
intersections the line bundle $S_b=S_a\otimes S_{ab},$ where $S_{ab}$ is a line bundle
over $U_a\cap U_b,$ induces a transition map by pull-back. Even though these
maps are not norm continuous on the Hilbert space completion to the fibre
$L^2$ spaces the algebras of compact operators patch to an Azumaya bundle
completing the smoothing algebra 
\begin{equation}
\Psi^{-\infty}_J\subset\cA_J.
\label{IBG.36}\end{equation}
\begin{proposition}\label{IBG.29} As a bundle of algebras over the base
of a small gerbe, $\cA_J$ is a smooth Azumaya bundle with Dixmier-Douady
class the 3-class represented by the gerbe.
\end{proposition}

\begin{proof} The Dixmier-Douady class of $\cA_J$ is the class of the gerbe, since
 in terms of a trivialization it is given by the sections of the $S_{ab}.$
\end{proof}

It follows that if $\Gamma \longrightarrow M$ is any bundle of classifying groups
for odd K-theory, for instance the groups of invertibles of the form
$\Id+A$ with $A\in\Psi^{-\infty}_\psi(X,V)$ for any other smooth
finite-dimensional fibre bundle $\psi:X\longrightarrow M$ (and vector bundle $V$
over it), then the bundle of groups
\begin{equation}
G^{\infty}_J=\{\Id+A;\text{invertible}\},\ A\in\CI(F^{[2]}\times_Y\Gamma
^{[2]};J\otimes\hom(V))
\label{IBG.30}\end{equation}
is a classifying bundle for twisted odd K-theory of $Y$
\begin{equation}
K^1(Y;\cA_J)=[Y;G^\infty_J].
\label{IBG.31}\end{equation}

Since the sections of the fibre operators on $\Gamma$ can be approximated
by sections valued in a finite-dimensional subbundle, \eqref{IBG.31} is a
corollary of
 
\begin{lemma}\label{22.4.2022.2} If $\psi: \Gamma \longrightarrow Y$ is any fibre
  bundle and $\cA_J$ is a small Azumaya bundle over $Y$ then the the odd
  K-theory of $Y,$ twisted by $\cA_J,$ $K^1(Y;\cA_J)$ is
  represented by equivalence classes of sections defined in terms of the
  Morita-stabilized, pulled-back bundles
\begin{equation}
\alpha\in\CIc(\Gamma;\psi^*\Psi^{-\infty}_J\otimes M(N)))
\label{22.4.2022.3}\end{equation}
such that $\Id+\alpha$ is invertible, with equivalence up to stabilization
and homotopy. 
\end{lemma}

\begin{proof} By construction of the completion $\cA_J$ of the small
  Azumaya bundle, sections of $\cA_J$ can be norm-approximated by sections
  of $\Psi^{-\infty}_J$ and hence \eqref{22.4.2022.3} reduces to the
  definition of odd twisted K-theory.
\end{proof}

There is a similar realization of even twisted K-theory in which $\alpha$
in \eqref{22.4.2022.3} is replaced by 
\begin{equation}
\beta \in\CIc(\Gamma;\psi^*\Psi^{-\infty}_J\otimes \left(M(N))\oplus (M(N)\right))\Mst
(P_N+\beta )^2=P_N+\beta
\label{IBG.96}\end{equation}
where $P_N$ is the constant projection onto the first factor
$M(N).$ Then again $K^0(Y;\cA_J)$ corresponds to equivalence classes under
homotopy and stabilization.

\section{Semiclassical Atiyah-Singer map}

For the orientation of the reader we briefly recall the semiclassical
version of the Atiyah-Singer construction from \cite{MMS2}; it is a model
for the proof below.

For a smooth fibre bundle, $\phi:F\longrightarrow Y,$ with compact fibres
over a smooth manifold, $Y,$ the index map defined in terms of
pseudodifferential operators may instead be defined in terms of
semiclassical pseudodifferential operators, and even more conveniently in
terms of semiclassical smoothing operators.

Recall that the bundle of smoothing operators on the fibres of $F$ may be
defined directly in terms of the kernels on the fibre product 
\begin{equation}
\begin{gathered}
\Psi^{-\infty}_{\phi}(F;V)=\CI(F^{[2]};\pi_L^*V\otimes\pi_R^*V'\otimes\pi_R^*\Omega _{\fib})\\
\xymatrix{
&F^{[2]}\ar[dl]^{\pi_L}\ar[dr]_{\pi_R}\\
F&&F.}
\end{gathered}
\label{IBG.128}\end{equation}
Here $V$ is a complex vector bundle over $F$ and $\Omega _{\fib}$ is the
bundle of smooth densities on the fibres. Acting on smooth sections of
$V$ by fibre integration 
\begin{equation}
\begin{gathered}
\Psi^{-\infty}_{\phi}(F;V)\times \CI(F;V)\longrightarrow \CI(F;V)\\
(A,v)\longmapsto Av=(\pi_L)_*(A\pi_R^*v)
\end{gathered}
\label{IBG.129}\end{equation}
these form an algebra.

We are particularly interested in the group of invertible perturbations of
the identity 
\begin{equation}
\cG^{-\infty}_{\phi}(F;V)=\{A\in\Psi^{-\infty}_{\phi}(F;V);\Id+A\text{ is invertible}\}.
\label{IBG.160}\end{equation}

Rather than the symbolic representation of (even) K-theory on the
fibre-wise cotangent bundle $T^*_{\phi}F,$ which arises in the usual
Atiyah-Singer map, the relevant representations here are the more obvious ones in
terms of
\begin{equation}%
\begin{gathered}
\dcG(\overline{T^*_{\phi}F};M(N))=\{\Id+a\text{
  invertible};a\in\dCI(\overline{T^*_{\phi}F};M(N))\}\\
K^1(T^*_{\phi}F)=\cup_N \dcG(\overline{T^*_{\phi}F};M(N))/\simeq\Mand\\
\dcP(\overline{T^*_{\phi}F};M(2N))=\{P=P_N+b;b\in\dCI(\overline{T^*_{\phi}F};M(2N));P^2=P\}\\
K^0(T^*_{\phi}F)=\cup_N\dcP(\overline{T^*_{\phi}F};M(2N))/\simeq.
\end{gathered}
\label{IBG.130}\end{equation}
Here the Schwartz functions on the fibres of $T^*_{\phi}F$ are identified
as those smooth functions on the radial compactification which vanish to
infinite order at the boundary. In the even case, $P_N$ is the projection
onto the first $N\times N$ matrix subalgebra. In both cases the equivalence
relation is the combination of smooth homotopy and stability under direct sum.

Quantization of the representatives in \eqref{IBG.130} is through the bundle
of semiclassical smoothing algebras on the fibres of $F.$ As in
\eqref{IBG.128} these can be defined directly in terms of smooth kernels
but now on the space
\begin{equation}
F[2,\scl]=[F^{[2]}\times[0,1]_\epsilon;\Diag\times\{0\}]\overset{\beta
}\longrightarrow F^{[2]}\times[0,1]
\label{IBG.131}\end{equation}
where the diagonal has been blown up at parameter value $0.$ This space has three
boundary hypersurfaces, the unchanged boundary $H_1=\{\epsilon =1\},$ the `front
face', $\ff,$ produced by the blow-up and the lift, or proper transform
$H_0=\beta ^\#\{\epsilon =0\}.$ The latter two together constitute the preimage
of $\{\epsilon =0\}$ under the blow-down map $\beta.$ There is a natural
diffeomorphism
\begin{equation}
\ff(F[2,\scl])\longrightarrow \overline{T_{\phi}F}
\label{IBG.133}\end{equation}
to the radial compactification of the fibre-wise tangent bundle of $F.$ 

Then the semiclassical smoothing operators are given by the space of
kernels 
\begin{equation}
\Psi^{-\infty}_{\phi,\scl}(F;V)=\epsilon
^{-d}\{A\in\CI(F[2;\scl];\pi_L^*V\otimes\pi_R^*V'\otimes\pi_R^*\Omega
_{\phi});A=0\Mat H_0\}.
\label{IBG.132}\end{equation}
Here the bundles are further lifted to $F[2,\scl]$ under $\beta$ and the
singular factor, $\epsilon ^{-d}$ (where $d$ is the fibre dimension of $F$) really
corresponds to the behaviour of the measure. Note that in local coordinates,
with bundle coefficients ignored, the kernels are of the form $\epsilon
^{-d}A(x,\frac{y-y'}\epsilon)$ where $x$ is a coordinate in the base, $y$
and $y'$ are local fibre coordinates and $A$ is smooth and Schwartz in the
last, Euclidean, variables.

Again these form an algebra of operators on $\CI(F\times[0,1])$
  acting on the fibres and with $\epsilon\in[0,1],$ as a (singular) parameter. The
  semiclassical symbol corresponds to the restriction of the leading term
  (i.e.\ after removal of $\epsilon ^{-d})$ at the front face, followed by
  fibre Fourier transform
\begin{equation}
\cF:\dCI(\overline{T_{\phi}F};\Omega)\longleftrightarrow \dCI(\overline{T^*_{\phi}F})
\label{IBG.158}\end{equation}
and gives a short exact, multiplicative, sequence
\begin{equation}
\xymatrix{
\epsilon \Psi^{-\infty}_{\phi,\scl}(F;V)\ar[r]&
\Psi^{-\infty}_{\phi,\scl}(F;V)\ar[r]^-{\sigma _{\scl}}
&\dCI(\overline{T^*_{\phi}F};\hom V).
}
\label{IBG.134}\end{equation}

The two spaces of semiclassical operators 
\begin{equation}
\begin{gathered}
\cG^{-\infty}_{\phi,\scl}(F;M(N))=\{\Id+A\text{ invertible };A\in \Psi^{-\infty}_{\phi,\scl}(F;M(N)\}\\
\cP^{-\infty}_{\phi,\scl}(F;M(2N))=\{P_0+B\text{ idempotents};B\in
\Psi^{-\infty}_{\phi,\scl}(F;M(2N)\} 
\end{gathered}
\label{IBG.135}\end{equation}
can be used in place of the algebras of pseudodifferential operators in the
usual approach.

To define the semiclassical version of the index map observe that the K-theory of the base
may be realized in terms of smoothing operators on the fibres of any
non-trivial bundle such as $F.$ For odd K-theory
\begin{equation}
K^1(Y)=\cG^{-\infty}_\phi(F;M(N))/\simeq.
\label{IBG.99}\end{equation}
Similarly for even K-theory
\begin{equation}
K^0(Y)=\{B=P_N+b,\ b\in \Psi^{-\infty}_\phi(Y;M(2N));B^2=B\}/\simeq.
\label{IBG.100}\end{equation}
Here the equivalence relation can be taken to be (smooth) homotopy.
The result is independent of $N$ via a stability argument; the trivial
bundles may be replaced by vector bundles over $F.$

The identifications \eqref{IBG.99} and \eqref{IBG.100} correspond to
appropriate density statements for finite rank operators in smoothing
operators. Namely given any invertible family $\Id+A,$ with $A\in
\Psi^{-\infty}_\phi(F;M(N)),$ there is a trivial finite rank subbundle with
idempotent projecting onto it (given by a finite rank family of smoothing operator)
\begin{equation*}
E\subset \CI(F;V\otimes\bbC^N),\ \Pi_E:\CI(F;V\otimes\bbC^N)\longrightarrow \CI(Y;E)\subset
\CI(F;V\otimes\bbC^N)
\label{IBG.138}\end{equation*}
such that $\Id+A$ is homotopic through invertible families to $\Id+A'$
where $P_EA'P_E=A'.$ There is a similar reduction for even K-theory to
finite rank idempotents commuting with $P_N.$

Now the index map in odd K-theory corresponds to the two maps 
\begin{equation}
\xymatrix{
&\dcG(\overline{T^*_{\phi}F};M(N))\ar[r]& K^1(T^*_{\fib}F)\ar@{-->}[dd]^{\ind_{\scl}}\\
\cG^{-\infty}_{\phi,\scl}(F;M(N))\ar@{->>}[ur]^{\sigma _{\scl}}\ar[dr]_{\epsilon =1}\\
&\cG^{-\infty}_{\phi}(F;M(N))\ar[r]&K^1(Y).
}
\label{IBG.136}\end{equation}
\begin{proposition}\label{IBG.173} For the group of invertible
  perturbations of the identity the symbol map, for any vector bundle $V$
  over $F$
\begin{equation}
\cG^{-\infty}_{\phi,\scl}(F;V)\longrightarrow \dcG(\overline{T^*_{\phi}F};\hom(V))
\label{IBG.174}\end{equation}
is an homotopy equivalence.
\end{proposition}

\begin{proof} The symbol map is surjective at the level of algebras in
  \eqref{IBG.134}. Moreover the subspace which can be written somewhat
  informally as
\begin{equation}
\epsilon ^\infty\Psi^{-\infty}_{\phi,\scl}(F;V)=\{A\in\CI(F[2;\scl];\pi_L^*V\otimes\pi_R^*V'\otimes\pi_R^*\Omega
_{\phi});A=0\Mat H_0\cup\ff\}
\label{IBG.175}\end{equation}
is identified by the blow down map with the smoothly parameterized elements
of $\Psi^{-\infty}_{\phi}(F;V)$ which vanish to infinite order at $\epsilon
=0.$ These form an ideal in $\Psi^{-\infty}_{\phi,\scl}(F;V)$ with quotient 
\begin{equation}
\Psi^{-\infty}_{\phi,\scl}(F;V)/\epsilon
^\infty\Psi^{-\infty}_{\phi,\scl}(F;V)=\bigoplus_{k\ge0} \epsilon
^k\dCI(\overline{T^*_\phi F};\hom V)
\label{IBG.176}\end{equation}
with a star product with initial multliplcative term. This is essentially
the same `Moyal' star product as the symbolic product of pseudodifferential
operators. 

It follows that an operator of the form $\Id+A$ with invertible symbol is
invertible for small $\epsilon >0.$ Namely composition with an operator
with symbol $\sigma _{\scl}(\Id+A)^{-1}$ gives $\Id+\epsilon B$ by
\eqref{IBG.134}. Iteration allows the error to be improved to $O(\epsilon
^\infty)$ and the the Neumann series converges to give an invertible
operator for $\epsilon <\epsilon _0.$ Since the parameter in $\epsilon >0$
can be freely rescaled the surjectivity of the symbol map at the group
level follows. Conversely it follows that any two operators with the same
invertible symbol can be smoothly deformed using the star product, to
differ by an element of the ideal and then to be equal for small $\epsilon
>0.$ Then by retracting in the parameter $\epsilon $ towards $\epsilon =1$
they can be connected by a smooth curve in the group.
\end{proof}

So this allows any symbol to be `quantized' to an invertible family which
is well-defined up to homotopy. In view of \eqref{IBG.99}, this
`quantization' defines a map into $K^1(Y).$ Moreover homotopy and stability
at the symbolic level lifts to the group of invertible semiclassical
operators so the index map on the right is well-defined by stabilizing over
$N.$ Indeed the freedom to deform and reparameterize a semiclassical family
in $\epsilon >0$ means that one can arrange that the quantization of a
given invertible symbol, i.e.\ an invertible family with this as
semiclassical symbol, becomes by a family of finite rank smoothing family at
$\epsilon =1.$

The even case is similar. Then the direct analogue of the Atiyah-Singer's quantization by
pseudodifferential operators holds.

\begin{theorem}\label{IBG.137} The semiclassical index map in \eqref{IBG.136} is given by
  push-forward in K-theory.
\end{theorem}

In fact, assuming the Atiyah-Singer index theorem, this can be proved by
deformation back to the pseudodifferential picture. Since we wish to
generalize this result to the twisted case, we recall from \cite{MMS2} a direct
semiclassical proof, which parallels that of Atiyah and Singer.

\section{Extensions of semiclassical quantization}\label{S.Ext}

We need to generalize the semiclassical index map in several respects to
make the proof through embedding run smoothly. More detail on these
constructions can be found in \cite{GenProd}.

The first such generalization is to allow the operators to take values in
the smoothing operators on a second bundle. Thus let
$\gamma:G\longrightarrow F$ be a compact fibre bundle over $F,$ so
$\phi=\psi\circ\gamma$ fibres $G$ over $Y:$
\begin{equation}
\xymatrix{
G\ar[r]^-{\gamma}\ar[dr]_-{\phi}&F\ar[d]^-{\psi}\\
&Y.
}
\label{IBG.151}\end{equation}
Then the smoothing operators on the
fibres of $G$ as a bundle over $Y$ have kernels which are smooth on the
fibre product
\begin{equation}
G^{[2]}\longrightarrow Y.
\label{IBG.146}\end{equation}
These may be considered as smooth sections over $F^{[2]}$ of the smoothing
operators on the fibres of $G.$ At a point $(x,y,y')\in F^{[2]}$
these take values in $\CI(G_y\times G_{y'}),$ the smoothing operators between
the two fibres of $G.$

Thus the smoothing operators on the fibres of $G$ as a bundle
over $Y$ can be interpreted as smoothing operators on the fibres of $F$
with values in the modules of smoothing operators between the fibres of $G$
as a bundle over $Y$ 
\begin{equation}
\Psi^{-\infty}_{\phi}(G;W)=\Psi^{-\infty}_{\psi}(F;\Psi^{-\infty}_{\gamma}(G;W))
\label{IBG.147}\end{equation}
where $W$ is a bundle over $G$ and the notation does not quite reflect the
fact that the image space on the right is the bundle of modules.

The fibre diagonal of $F$ lifts to an embedded submanifold of $G^{[2]}$
fibering over $Y$ so the semiclassical calculus on the fibres of $F$ `with
values in the smoothing operators on $G$' can be defined directly as the
space of kernels on
\begin{equation}%
\begin{gathered}
F[2,\scl;G]=[G^{[2]}\times[0,1];(\gamma^{[2]})^{-1}\Diag_F\times\{0\}],\\
\Psi^{-\infty}_{\psi,\scl}(F;\Psi^{-\infty}_{\gamma}(G;W))=\{A\in\CI(F[2,\scl;G];A\equiv
0\Mat\beta^{\#}\{\epsilon =0\}\}\\
\sigma_{\scl}:\Psi^{-\infty}_{\psi,\scl}(F;\Psi^{-\infty}_{\gamma}(G;W))\longrightarrow \dCI(\overline{T^*_{\fib}F};\Psi^{-\infty}_{\gamma}(G;W)).
\end{gathered}
\label{IBG.148}\end{equation}
Here the, surjective, semiclassical symbol is defined over the fibre diagonal of $F^{[2]}$
over $Y$ and so does take values in the sections of the bundle of smoothing
algdbras on the fibres of $G$ over $F.$

The index map in the odd case then simplifies since for the group 
\begin{equation}
\dcG(\overline{T^*_{\phi}F};\Psi^{-\infty}_{\gamma}(G,W))=\{\Id+a\text{ invertible};a\in\dCI(\overline{T^*_{\phi}F};\Psi^{-\infty}_{\gamma}(G,W))\}
\label{IBG.159}\end{equation}
to define the odd K-theory of $T^*_{\phi}F$ stabilization is not required.
So letting $\cG^{-\infty}_{\phi,\scl}(F;\Psi^{-\infty}_{\gamma}(G;W))$ be
the corresponding group of invertibles gives the commutative diagramme
\begin{equation}
\xymatrix{
&\dcG(\overline{T^*_{\phi}F};\Psi^{-\infty}_{\gamma}(G,W))\ar@{>>}[r]& K^1(T^*_{\fib}F)\ar@{-->}[dd]^{\ind_{\scl}}\\
\cG^{-\infty}_{\phi,\scl}(F;\Psi^{-\infty}_{\gamma}(G,W))\ar@{->>}[ur]^{\sigma _{\scl}}\ar[dr]_{\epsilon =1}\\
&\cG^{-\infty}_{\phi}(G;W)\ar[r]&K^1(Y).
}
\label{IBG.149}\end{equation}

That this is the same map as in \eqref{IBG.136} follows from the existence of
embeddings of finite rank bundles in $\Psi^{-\infty}_{\gamma}(G,W),$ as in
\eqref{IBG.99}. The even case is similar.

The second generalization is to allow a degree of non-compactness of the
fibres of the fibration, $F.$ Since the simple properties of the
semiclassical calculus depend on the existence of a functional calculus we
restrict attention to the the case of fibres which are the interiors of
compact manifolds with corners
\begin{equation}
\xymatrix{
F\ar@{^(->}[r]\ar[dr]_{\phi}&\tilde F\ar[d]^{\phi}\\
&Y.
}
\label{IBG.142}\end{equation}
Then as the smoothing algebra we take 
\begin{equation}
\dot\Psi^{-\infty}_{\phi}(F;V)=\dCI(\tilde
F^{[2]};\pi_L^*V\otimes\pi_R^*V'\otimes\Omega _{\fib})
\label{IBG.143}\end{equation}
where the notation indicates that the kernels are required to vanish to
infinite order on both boundaries of the fibres of $\tilde F^{[2]}.$ This
algebra has essentially the same properties as the boundaryless case.

Defining the semiclassical smoothing calculus for fibres which are the
interiors of manifolds with coners requires a further modification, since
the fibre diagonal of $\tilde F^{[2]}$ does not pass transversally through
the boundary. So to be in a position to blow up the diagaonal at $\{\epsilon
=0\}$ we first resolve the intersection of this submanifold with the
boundary. Since $\Diag_{\tilde F}$ is naturally diffeomorphic to $\tilde F$
this consists of the embedded submanifolds $B\times\{0\}\subset
\Diag_{\tilde F}\times\{0\}$ where $B\in\cM_*(\tilde F)$ runs over the
boundary faces (including hypersurfaces) of $\tilde F.$ After blowing these
up, in order of increasing dimension, the lift of $\Diag_{\tilde
  F}\times\{0\}$ meets the boundary transversally, so we can define
\begin{equation}
\tilde F[2,\scl]=[\tilde F^{[2]};\cM_*(\tilde F)\times\{0\};\Diag_{\tilde F}\times\{0\}].
\label{IBG.150}\end{equation}
The `front face', $\ff,$ is the boundary hypersurface introduced by the last blow-up.

As the semiclassical smoothing operators in this case we take the kernels 
\begin{equation}%
\begin{gathered}
\tilde F[2,\scl]=[(F^{[2]})_{\tb}\times[0,1];\Diag\times\{0\}],\\
\begin{aligned}
\Psi^{-\infty}_{\phi,\scl}(\tilde F;V)=\epsilon
^{-d}\{A\in\CI(\tilde F[2;\scl];&\pi_L^*V\otimes\pi_R^*V'\otimes\pi_R^*\Omega
_{\phi});\\
&A=0\Mat\pa\tilde F[2;\scl]\setminus(\ff\cup\{\epsilon =1\})\}.
\end{aligned}
\end{gathered}
\label{IBG.145}\end{equation}
So, by assumptions, these kernels vanish to infinite order all of the boundary
hypersurfaces introduced by the resolution in \eqref{IBG.150}, except for
the last step and at all the `old boundaries' of $\tilde F^{[2]}.$ This
ensures that in $\epsilon >0$ the kernels reduce to the smoothing operators
in \eqref{IBG.143} and that they form an algebra as before. 

In fact there is no obstruction to combining these two extensions, nor to
allowing the base to be the interior of a compact manifold with corners.

The third generalization is to a doubly-semiclassical algebra. Thus again
we consider an iterated fibration of compact manifolds as in \eqref{IBG.151}. Then
the fibre product of $G$ over $Y$ is 
\begin{equation}
G^{[2]}_\phi=\{(p,p')\in G^2;\phi(p)=\phi(p')\}.
\label{IBG.152}\end{equation}
The image of this under $\gamma \times\gamma :G^2\longrightarrow F^2$ is
the fibre product of $F$ with itself over $Y,$ $F^{[2]}_{\psi},$ and
$G^{[2]}_{\phi}\longrightarrow F^{[2]}_{\psi}$ is a fibration with fibre
$Z\times Z$ where $Z$ is the typical fibre of $\gamma.$ Within
$G^{[2]}_{\phi}$ lie the two fibre diagonals 
\begin{equation}
\Diag_G\subset\Diag_F^G\subset G^{[2]}_{\phi}
\label{IBG.153}\end{equation}
where $\Diag_F^G$ is the preimage of the diagonal in $F.$

The `doubly semiclassical space' we take is 
\begin{equation}
G[2;\text{2-scl}]=
\big[G^{[2]}_{\phi}\times[0,1]_{\epsilon}\times [0,1]_{\delta};
\Diag_G\times\{\epsilon =0\};\Diag_F^G\times\{\delta =0\}\big].
\label{IBG.154}\end{equation}
Under the first blow-up the remaining centre lifts to the diagonal, in
$\delta =0,$ of the adiabatic space discussed above, so the second blow-up is well-defined.

The resulting manifold has six boundary hypersurfaces, which we denote as
$\ff_G,$ $\ff_F,$ $H_{\epsilon =0},$ $H_{\delta =0},$ $H_{\epsilon =1}$
and $H_{\delta =1};$ here the last four are really the proper transforms of
the boundary hypersurfaces of $G^{[2]}_{\phi}\times[0,1]^2.$ We define the
doubly-semiclassical operators acting on sections of a bundle
$W\longrightarrow G$ through the space of kernels
\begin{equation}
\Psi^{-\infty}_{\text{2-scl}}(G;W)=\epsilon
^{-d}\delta ^{-q}\big\{A\in\CI(G[2;\text{2-scl}];
\pi_L^*W\otimes\pi_R^*W'\otimes\pi_R^*\Omega_{\psi};A\equiv0\Mat
H_{\delta=0}\cup H_{\epsilon=0}\big\}
\label{IBG.155}\end{equation}
where $d$ and $q$ are the fibre dimensions.  This is an algebra with the
product determined by composition of smoothing operators on the fibres of
$G$ over $Y$ away from the boundary faces. This can be seen by local
computation as for the
semiclassical case but is discussed abstractly in \cite{GenProd}.

We use this doubly semiclassical algebra, in Prositition~\ref{IBG.140}
below, to show how the index iterates. The index map for the combined
fibration corresponds to $\delta =1,$ the index map for the first
fibration (lifted to the fibre cotangent bundle of the second) corresponds
to $\delta =0$ and the second index map corresponds to $\epsilon =1.$

The front face for the first blow-up in \eqref{IBG.154} is the product of
the adiabatic front face for $\phi$ and the interval $[0,1]_{\delta}.$
It is then blown up at $\delta =0$ by the second blow up, but remains
contractible to $T_{\phi}G.$
In the second blow-up \eqref{IBG.155} the centre is the lifted diagonal for
the adiabatic space in $\delta =0.$ Thus
\begin{equation}
\ff_F\equiv\overline{T_{\psi}F}\times_{\psi}G[2;\ad]
\label{IBG.161}\end{equation}
is the adiabatic space lifted to the fibre tangent space for $F$ (rescaled
at $\epsilon =0$). 

As a result, the leading terms in the kernels at the four boundary
hypersurfaces, where they are not required to be trivial, define four
homomorphisms
\begin{equation}
\begin{gathered}
N_{\delta  =0}:\Psi^{-\infty}_{\text{2-scl}}(G;W)\longrightarrow
\Psi^{-\infty}_{\gamma,\scl}(G\times_\gamma\overline{T^*F};W),\\
N_{\delta =1}:\Psi^{-\infty}_{\gamma ,\scl}(G;W)\longrightarrow
\Psi^{-\infty}_{\psi,\scl}(F;\Psi^{-\infty}_{\gamma}(G;W))\\
N_{\epsilon=1}:\Psi^{-\infty}_{\text{2-scl}}(G;W)\longrightarrow
\Psi^{-\infty}_{\phi,\scl}(G;W)\\
\sigma _{\text{2-scl}}:\Psi^{-\infty}_{\text{2-scl}}(G;W)\longrightarrow\dCI(\ff_G;\hom{W}).
\end{gathered}
\label{IBG.156}\end{equation}
Other than $N_{\delta =1}$ these are surjective.

Thus the range of the first `normal operator' is the space of semiclassical
smoothing operators for the pull-back of $\gamma:G\longrightarrow F$ to
$\overline{T^*F}.$ The range of the second is the semiclassical algebra for
$\psi:F\longrightarrow Y$ with values in the smoothing operators on the
fibres of $G$ over $F,$ as discussed above. The range of the third normal
operator is the semiclassical algebra for $G$ as a bundle over $Y.$ Finally
the last `semiclassical symbol' maps surjectively to the Schwartz functions
on the manifold with corners $\ff_G$ in \eqref{IBG.161}. From the
construction it follows that for the associated groups the restriction maps
\begin{equation}
\begin{gathered}
\big|_{H_{\epsilon =0}}:\dcG(\ff_G;\hom{W})\longrightarrow
\dcG(\overline{T_\phi G};\hom(W))\Mand\\
\big|_{H_{\delta =0}}:\dcG(\ff_G;\hom{W})\longrightarrow
\dcG(\overline{T_\gamma G\times_F T_{\psi}F};\hom(W))
\end{gathered}
\label{IBG.157}\end{equation}
are homotopy equivalences. 

\section{Properties of the index map}

To carry this `embedding' proof through we first check three properties of
semiclassical quantization.

\begin{proposition}[Extension]\label{IBG.139} If $\psi_i:F_i\longrightarrow
  Y,$ $i=1,2,$ are bundles of manifolds with corners over $Y$ of the same
  fibre dimension and the fibres of $F_1$ are contained in the interiors of
  the fibres of $F_2$ then the semiclassical index maps are consistent
  giving a commutative diagramme
\begin{equation}
\xymatrix{
K^*(T^*_{\psi_1}\inn F_1)\ar@{^(->}[rr]\ar[dr]_{\ind_{\scl}}&&K^*(T^*_{\psi_2}\inn F_2)\ar[dl]^{\ind_{\scl}}\\
&K^*(Y)
}
\label{IBG.165}\end{equation}
\end{proposition}

\begin{proof} For smoothing operators
  $\Psi^{-\infty}_{\psi_1}(F_1;W)\subset\Psi^{-\infty}_{\psi_2}(F_2;W),$
  for any bundle $W$ over $F_2$ corresponds to the inclusion of the smooth
  kernels with support in $F_1^{[2]}\subset F_2^{[2]}.$ Essentially the
  same is true for the semiclassical spaces, except that there is no
  immediate inclusion of the resolved space. That is, the semiclassical
  smoothing operators for $F_1$ may be identified with the semiclassical
  operators for $F_2$ with the corresponding support property in $\epsilon
  >0.$ In the construction of $F_1[2,\psi_1\scl],$ as in \eqref{IBG.150},
  the boundary faces of the diagonal are blown up first. In
  $F_2^{[2]}\times[0,1]$ these boundary faces are all submanifolds of
  $\Diag_{F_2}\times\{0\}$ so they can be blown up \emph{after}
    $\Diag_{F_2}$ in the construction of $F_2[2;\psi_2\scl].$ It follows
      that  
\begin{equation}
F_1[2;\psi_1\scl]\subset [F_2[2,\psi_2\scl];\cM_*(\tilde F)\times\{0\}]
\label{IBG.166}\end{equation}
and the support property gives the identification  
\begin{equation}
\Psi^{-\infty}_{\psi_1,\scl}(F_1;W)\subset\Psi^{-\infty}_{\psi_2,\scl}(F_2;W)
\label{IBG.167}\end{equation}
from which the consistency of the index maps follows.
\end{proof}

\begin{proposition}[Iteration]\label{IBG.140} For a double fibration
  \eqref{IBG.151} of compact manifolds with corners the index map for
$\phi$ is the composite of the index maps for
  $\widetilde{\gamma}:G\times_F\overline{T^*_{\psi}F}\longrightarrow T^*_{\psi}F$ and $\psi.$ 
\end{proposition}
\noindent Note that the index map for $\widetilde{\gamma }$ maps from the
K-theory of the fibre cotangent bundle, which is the fibre product over $F$
of $T^*_{\gamma}G$ and $T^*_{\psi}F;$ this is diffeomorphic to $T^*_{\phi}G.$

\begin{proof} The doubly semiclassical algebra discussed above is set up
  for this purpose. Namely we consider the group of invertible
  perturbations of the identity in this algebra
\begin{equation}
\cG^{-\infty}_{\operatorname{2-scl}}(G;M(N))=\left\{A\in\Psi^{-\infty}_{\text{2-scl}}(G;M(N));
\Id+A\text{ is invertible}\right\} 
\label{IBG.162}\end{equation}
for the operators as in \eqref{IBG.155}. This leads to a commutative
diagramme enlarging \eqref{IBG.136}
\begin{equation}
\xymatrix{
\dcG(T^*_{\gamma}G\times_FT^*_{\psi}F;M(N))\ar@/_5pc/@{-->}[dd]_{\ind_{\gamma\scl}}
&\dcG(\ff_G;M(N))\ar[r]^-{\delta=1 }
\ar[l]_-{\delta =0 }&\dcG(T^*_{\phi}G;M(N))
\ar@/^4pc/@{-->}[ddd]_{\ind_{\phi\scl}}\\
\cG^{-\infty}_{\gamma,\scl}(G\times_\gamma\overline{T^*F};W)\ar[u]_{\sigma _{\scl}}\ar[d]^{\epsilon =1}&
\cG^{-\infty}_{\operatorname{2-scl}}(G;M(N))\ar@{>>}[r]^{N_{\delta=1}}\ar[u]_{\sigma _{\text{2-scl}}}
\ar@{>>}[l]_-{N_{\delta =0}}\ar[d]^-{N_{\epsilon  =1}}&\cG^{-\infty}_{\phi,\scl}(G;W)\ar[dd]^{\delta =\epsilon  =1}
\ar[u]_{\sigma _{\scl}}\\
\dcG(\overline{T^*_{\psi}F};\Psi^{-\infty}_{\gamma}(G;W))
\ar@{-->}[drr]_{\ind_{\psi\scl}}&
\cG^{-\infty}_{\psi,\scl}(F;\Psi^{-\infty}_{\gamma}(G;W))\ar[dr]^{\epsilon =\delta =1}\ar[l]_{\sigma _{\scl}}&\\
&&\cG(Y;\Psi^{-\infty}_{\phi}(G;M(N)))
}
\label{IBG.164}\end{equation}

The commutativity of this diagramme identifies the semiclassical index
map for $\phi,$ which maps from stable homotopy classes on the top right to
the bottom right, with the composite of the two similar maps from top left
to bottom left and from bottom left on to bottom right.
\end{proof}

\begin{proposition}[Thom isomorphism]\label{IBG.141} If $\mu:V\longrightarrow Y$ is a real
  vector bundle over a compact manifold then the semiclassical index
  map for $\mu$ is an isomorphism and a Thom element,
  $P=P_N+q\in\dcP(\overline{V+V'};M(2N))$ with semiclassical index a trivial line
  bundle, generates the inverse to the index map in odd K-theory through 
\begin{multline}
\cG(Y;M(L))\ni a \longmapsto\\
(P_N\otimes a^{-1}+(\Id-P_N)\otimes \Id)(P\otimes a+(\Id-P)\otimes \Id)
\in \dcG(\overline{V+V'};M(2NL)).
\label{IBG.168}\end{multline}
\end{proposition}
\noindent Of course, the only point here is that the Thom
isomorphism from the fibre cotangent bundle, which is $V\oplus V',$ is
realized by the semiclassical index.

\begin{proof} In case $V$ is a trivial 1-dimensional real bundle one can
  construct a Thom, i.e.\ Bott, class explicitly using Pauli matrices and
  the complexified Clifford algebra (cf. \cite{ABS} and with connection \cite{MQ}). That this has a trivial complex line
  bundle as index follows from a trace calculation giving the numerical
  index. That \eqref{IBG.168} gives an inverse to the odd index follows
  directly. Note that the first, normalizing, factor is fibre constant so
  quantizes trivially. Its presence ensures that this is a Schwartz
  perturbation of the identity.

Using (a simple case of) the iteration result above shows that the
semiclassical index map is an isomorphism for any trivial real bundle. For
a general bundle the iteration result applies to show that the composite
with the index map for a complement in a trivial bundle is an isomorphism,
so all such maps are surjective. Injectivity follows by an argument of
Atiyah, \cite{At68}.
\end{proof}

\section{Proof of Theorem~\ref{IBG.137}}

The structure of the proof is that given by Atiyah and Singer, identifying
the index map with the Gysin construction of the push-forward in
K-theory. We follow the odd index.

The fibration, $\phi:F\longrightarrow Y,$ of compact manifolds is first
embedded in a product fibration 
\begin{equation}
\xymatrix{
F\ar@{^(->}[rr]\ar[dr]_{\phi}&&Y\times \bbR^N\ar[dl]^\pi\\
&Y.
}
\label{IBG.169}\end{equation}
The collar neighbourhood theorem gives a smooth identification of an open
neighbourhood of the image of $F$ with the normal bundle, $V.$ The Thom
isomorphism identifies the K-theory of $T^*_{\phi}F$ with that of
$T^*_{\phi}F\times_F(V\oplus V').$ Propositions~\ref{IBG.140} and
\ref{IBG.141} show that the semiclassical index map for $\phi$ is the
product of the inverse of the Thom ismorphism, given by the semiclassical
index, for $V\oplus V'$ pulled back to $T^*_{\phi}F$ and the index map for
$V$ as a bundle over $Y.$ The extension result, Proposition~\ref{IBG.139}
then shows that this factors through the Thom isomorphsim for the trivial
bundle $\bbR^N$ over $Y.$ This identifies the semiclassical index with the
push-forward in K-theory.

\section{The twisted index map}\label{Tw-in-ma}

For a small bundle gerbe consider the `semiclassical space' associated to
$F$ in \eqref{IBG.131}.

\begin{lemma}\label{IBG.19} The lift of the simplicial bundle $J$ to
  $F[2,\scl]$ is trivialized over $\ff.$
\end{lemma}

\begin{proof} This is a standard consequence of the simplicial property of
  $J.$ Namely, over the diagonal in $F^{[2]}$ the simplicial trivialization
  reduces to an isomorphism
\begin{equation}
J_{(p,p)}\otimes J_{(p,p)}\longrightarrow J_{(p,p)}\Longrightarrow J_{(p,p)}=\bbC.
\label{IBG.20}\end{equation}
\end{proof}

Generalizing \eqref{IBG.132}, the algebra of twisted semiclassical smoothing
operators is defined as a space of kernels 
\begin{equation}
\Psi^{-\infty}_{\scl,J}(F)=\epsilon ^{-d}\{A\in\CI(F[2,\scl];\beta
^*J\otimes\Omega_R);A\equiv0\Mat H_0\}.
\label{IBG.21}\end{equation}

In $\epsilon >0$ the blow-up is irrelevant so these kernels define smooth
sections of $J$ over $F^{[2]}$ depending smoothly on $\epsilon >0.$

\begin{proposition}\label{IBG.22} The composition in $\epsilon >0$
  given by the small Azumaya bundle extends to a smooth algebra structure on
  $\Psi^{-\infty}_{\scl,J}(F).$
\end{proposition}

\begin{proof} The smoothness of the composition down to
  $\epsilon =0$ follows from the discussion in the proof of
  Proposition~\ref{IBG.29}. Namely, over the $U_a$ the algebra composition
  reduces to that for sections of the bundle $J_a$ and hence the twisted
  semiclassical algebra reduces to the standard semiclassical algebra for
  sections of a line bundle. It follows that the product extends smoothly
  down to $\epsilon =0$ globally.
\end{proof}

The symbol map for this semiclassical $J$-twisted algebra has the same
range as the untwisted algebra, since the symbol is defined from the
leading term of the kernel at the hypersurface $\ff$ where $J$ is
simplicially trivialized. Thus there is a short exact sequence of algebras
\begin{equation}
\xymatrix{
\epsilon \Psi^{-\infty}_{\scl,J}(F;E)\ar[r]& \Psi^{-\infty}_{\scl,J}(F;E)\ar[r]&
\dCI(\overline{T^*_{\fib}F};\hom E).
}
\label{IBG.24}\end{equation}

As in the untwisted case the smooth families
$[0,1]_{\epsilon}\longrightarrow \Psi_{J}(F;E)$ which vanish to infinite
order at $\epsilon =0$ form an ideal in the semiclassical algebra with the
quotient being a $*$-algebra. It follows that a semiclassical operator
$\Id+A$ with invertible symbol is invertible for small $\epsilon >0.$ Since
the parameter $\epsilon $ can be rescaled it follows that the symbol map
from the group of invertibles 
\begin{equation}
\sigma_{\scl,J}:\cG_{\scl,J}(F;E)=\{\Id+A;\text{ invertible }A\in\Psi^{-\infty}_{\scl,J}(F;E)\}
\longrightarrow \cG^{-\infty}(\overline{T^*_{\phi}F};\hom(E))
\label{IBG.170}\end{equation}
is again homotopy equivalence.

Thus, as in \eqref{IBG.136}, the twisted semiclassical
index is defined 
\begin{equation}
\xymatrix{
&\dcG(\overline{T^*_{\phi}F};M(N))\ar[r]& K^1(T^*_{\fib}F)\ar@{-->}[dd]^{\ind_{\scl,J}}\\
\cG^{-\infty}_{\scl,J}(F;M(N))\ar@{->>}[ur]^{\sigma _{\scl}}\ar[dr]_{\epsilon =1}\\
&\cG^{-\infty}_{J}(F;M(N))\ar[r]&K^1(Y;\cA_J).
}
\label{IBG.171}\end{equation}

\begin{theorem}\label{IBG.25} The twisted semiclassical index map defined
  by stabilization of \eqref{IBG.171} is the push-foward to twisted
  K-theory given by the simplicial trivialization of $J$ pulled back to
  $F^{[2]}:$ 
\begin{equation}
\ind_{\scl,J}:K^1(T^*_{\fib}F)\longrightarrow K^1(Y;\cA_J).
\label{IBG.172}\end{equation}
\end{theorem}

The even index map is defined analogously in terms of idempotents.

Theorem~\ref{IBG.25} is proved by combining the deformation to the more general
`range-twisted' index map, given by Proposition~\ref{IBG.108}, together with the
proof of the latter index theorem in Theorem~\ref{IBG.106}.

\section{Range-twisted quantization}

Although we have defined the twisted index map by quantization through the
bundle of semiclassical twisted smoothing operators it is also possible to
realize the same twisted index map by taking the range of the operators to
be twisted.

For a finite-dimensional fibre bundle $\psi:G\longrightarrow Y$ we can
consider the smoothing operators on the fibres of $G$ with values in the
small Azumaya bundle $\Psi^{-\infty}_J$ over $Y.$ This is geometrically the
same as the first of the extensions to semiclassical quantization discussed
in \S\ref{S.Ext} with (bundles reversed) but simplified in the sense that
the second bundle is pulled back from the base.

Thus \eqref{IBG.21} may be modified to give the kernels 
\begin{multline}
\Psi^{-\infty}_{\scl,\psi}(G;\Psi^{-\infty}_J\otimes M(N))\\
=\epsilon ^{-D}
\{A\in\CI(G[2,\scl]\times_YF^{[2]};J\otimes\Omega_R\otimes
M(N));A\equiv0\Mat\beta^\#(\{\epsilon =0\})\}
\label{IBG.97}\end{multline}
where $D$ is the fibre dimension of $G.$ That this is an algebra follows by
localization over $Y$ where $J$ is simplicially trivialized by a line
bundle over $F.$ The symbol map obtained by evaluating the leading term at the front face of
$G[2,\scl]\times_YF^{[2]}$ followed by fibre Fourier transform
\begin{equation}
\sigma _{\scl,J}:\Psi^{-\infty}_{\scl,\psi}(G;\Psi^{-\infty}_J\otimes M(N))\longrightarrow
\dCI(\overline{T^*_{\psi}}G\times_YF^{[2]};\Psi^{-\infty}_J\otimes M(N))
\label{IBG.102}\end{equation}
now takes values in the twisted smoothing operators. As before it gives the
multiplicative short exact sequence 
\begin{equation}
\xymatrix@C=1.5em{
\epsilon \Psi^{-\infty}_{\scl,\psi}(G;\Psi^{-\infty}_J\otimes M(N))\ar@{^(->}[r]&
\Psi^{-\infty}_{\scl,\psi}(G;\Psi^{-\infty}_J\otimes M(N))\ar[r]^-{\sigma _{\scl}}&
\dCI(\overline{T^*_{\psi}}G\times_YF^{[2]};\Psi^{-\infty}_J\otimes M(N)).
}
\label{IBG.103}\end{equation}

The arguments showing that for the corresponding group of invertible
perturbations of the identity
\begin{equation}
\cG^{-\infty}_{\scl,J}(G;\Psi^{-\infty}_J\otimes M(N))
=\{\Id+A,\ A\in \Psi^{-\infty}_{\scl,\psi}(G;\Psi^{-\infty}_J\otimes
M(N));\text{ invertible}\}
\label{IBG.104}\end{equation}
the symbol map is an homotopy equivalence 
\begin{equation}
\cG^{-\infty}_{\scl,J}(G;\Psi^{-\infty}_J\otimes M(N))\longrightarrow
\dcG(\overline{T^*_{\fib}G};\Psi^{-\infty}_J\otimes M(N))
\label{IBG.177}\end{equation}
carry through essentially unchanged.

Since $\Psi^{-\infty}_J$ is dense in the Azumaya bundle approximation shows
that any odd twisted K-class on $T^*_{\fib}G,$ where the twisting corresponds to
the small gerbe on $F,$ may be represented by such an invertible family 
\begin{equation}
\dcG(\overline{T^*_{\fib}G};\Psi^{-\infty}_J\otimes M(N))\longrightarrow
  K^1(T^*_{\fib}G;\psi^*\cA_J).
\label{IBG.101}\end{equation}
\begin{theorem}\label{IBG.106} The `range-twisted index map' is well
 defined by semiclassical quantization as above from the diagramme
 modifying \eqref{IBG.171}:
\begin{equation}
\ind_{\scl,J}:K^1(T^*_{\fib}G;\psi^*\cA_J)\longrightarrow K^1(Y;\cA_J)
\label{IBG.107}\end{equation}
and coincides with push-forward in twisted K-theory.
\end{theorem}

\begin{proof} Any two quantizations of $\Id+a$ are homotopic through
  invertibles for sufficiently  small $\epsilon >0.$
\end{proof}

If the two fibre bundles in Proposition~\ref{IBG.106} are the same, so
$G=F$ is the total space of the small gerbe defining $\cA_J$ then
$\psi^*\cA_J$ is untwisted as an Azumaya bundle over $G=F.$ In that case
\eqref{IBG.107} and \eqref{IBG.11} are maps on the same spaces.

\begin{proposition}\label{IBG.108} The index maps \eqref{IBG.107}, for
  $G=F,$ and \eqref{IBG.11}, arising from \eqref{IBG.107} are the same when
  the trivially twisted K-theory $K^1(T^*_{\fib}F;\psi^*\cA_J)$ is
  identified with $K^1(T^*_{\fib}F)$ using the symplicial isomorphism 
\begin{equation}
  \pi_{34}^*J\cong \pi_{12}^*J\otimes \pi_{13}^*J^{-1}\otimes \pi_{24}^*J\text{ over }F^{[2]}\times _YF^{[2]}.
\label{IBG.109}\end{equation}
\end{proposition}

\begin{proof} The isomorphism \eqref{IBG.109}, arising from the simplicial
  property of $J$ is multiplicative and gives an identification of algebras 
\begin{equation}
\Psi^{-\infty}_{\scl,\phi}(F;\Psi^{-\infty}_J)\equiv \Psi^{-\infty}_{\scl,J}(F;\Psi^{\infty}_{\psi}(F;J^{-1}))
\label{IBG.112}\end{equation}
with their symbol maps. Hence the index maps are the same.
\end{proof}

Thus, as already noted, the proof of Theorem~\ref{IBG.25} is reduced by
Proposition~\ref{IBG.108} to that of Theorem~\ref{IBG.106}.

\begin{proof}[Proof of Theorem~\ref{IBG.106}] The proof of the
  semiclassical version of the Atiyah-Singer theorem above has been
  structured in such a way that it applies almost verbatim to the
  range-twisted case.
\end{proof}

Note that in fact the same argument serves to prove Theorem~\ref{IBG.106}
in the more general case of a smooth Azumaya bundle over the base, $Y,$ of a
fibration. Any integral three class on $Y$ gives such a bundle through the
associated principal $\PU$ bundle.

\section{Chern character}\label{S.OC}

When the K-theory of a manifold is represented in terms of equivalence
classes of complex vector bundles, the Chern character map, to even
cohomology, is given in terms of the curvature of a connection on the
bundles. To pass to the twisted Chern character we first consider the Chern
character for infinite rank bundles as a model for untwisted K-theory.

Suppose $M$ is a connected compact manifold, $\phi:F\longrightarrow M$ is a
smooth fibre bundle with compact fibres and $V$ is a complex vector bundle
over $F.$ If $Z$ is a model for the fibres of $F$ then over each fibre, $V$
is isomorphic to a fixed bundle $W\longrightarrow Z.$ Then $\CI(F;V)$ is a
bundle over $M$ with model fibre $\CI(Z;W).$

Now consider the bundle $\CI(F;V)\oplus\CI(F;V)=\CI(F;V\otimes\bbC^2)$
  and infinite rank subbundles which are determined by idempotents 
\begin{equation}
\Pi=\Pi_-+q,\ q\in\Psi^{-\infty}_{\phi}(F;V\otimes\bbC^2),\ \Pi^2=\Pi,\ \Pi_-=%
\begin{pmatrix}0&0\\
0&\Id
\end{pmatrix}. 
\label{IBG.65}\end{equation}
Then
\begin{equation}
K^0(M)=\{\Pi\text{ as in \eqref{IBG.65}}\}/\text{ smooth homotopy.}
\label{IBG.66}\end{equation}
This follows from the fact that every idemptotent as in \eqref{IBG.65} is
smoothly homotopic to one of the form 
\begin{equation}
\Pi_-+q',\ q'\Pi_-=\Pi_-q'=%
\begin{pmatrix}0&0\\0&q_e-q_-'
\end{pmatrix},\
q'(\Id-\Pi_-)=%
\begin{pmatrix}q'_+&0\\0&0
\end{pmatrix}
\label{IBG.67}\end{equation}
where $q'_\pm\in\Psi^{-\infty}_\phi(F;V)$ are finite rank idempotents and
$q_-'q_e=q_e=q_eq_-'$ where $q_e$ is a finite rank projection onto the range of a
trivial subbundle $e:\CI(M;\bbC^N)\longrightarrow \CI(F;V).$ Then $\Pi$ is
mapped to the K-class 
\begin{equation}
[\Pi]=\Ran(q_+)\ominus\Ran(q_e-q'_-)\longmapsto  K^0(M).
\label{IBG.68}\end{equation}

To give the Chern character directly in terms of such idempotents we
proceed to  {describe the curvature form} of a connection on the range
of $\Pi.$

First a connection on $\CI(F;V)$ as a bundle over $M$ can be obtained by
combining a connection, $\nabla^V,$ on $V$ over $F$ with an Ehresmann
connection on $F$ as a fibre bundle over $M.$ The Ehresmann connection is
specified by a smooth linear lifting map
\begin{equation}
\CI(M;TM)\ni u\longrightarrow
\tau(u)\in\CI(F;TF),\ \phi_*\tau(u)(p)=u(\phi(p))\ \forall\ q\in F.
\label{IBG.69}\end{equation}
Then the induced connection is 
\begin{equation}
\nabla:\CI(F;V)\longrightarrow \CI(M;T^*M)\otimes_{\CI(M)}\CI(F;V),\
\nabla_u v=\nabla^V_{\tau(u)}v.
\label{IBG.70}\end{equation}

As in the finite-dimensional case this extends uniquely to an extended
connection 
\begin{multline}
\widetilde\nabla:\CI(M;\Lambda ^*)\otimes_{\CI(M)}\CI(M;V)\longrightarrow
\CI(M;\Lambda ^{*+1})\otimes_{\CI(M)}\CI(M;V),\\
\widetilde\nabla(\alpha
\wedge w)=d\alpha \wedge w+(-1)^k\alpha \wedge \widetilde\nabla w,\ \alpha
\in\CI(M;\Lambda ^k).
\label{IBG.71}\end{multline}
Written in terms of a local trivialization of $F,$ and reduction of $V$ to
$W$ locally over $M,$ this operator is of the form 
\begin{equation}
\nabla=d_M+\gamma
\label{IBG.77}\end{equation}
where $\gamma$ is a first-order fibre differential operator.

Now the curvature of $\nabla$ is  
\begin{equation}
R=\widetilde\nabla^2:\CI(M;\Lambda ^*)\otimes_{\CI(M)}\CI(M;V)\longrightarrow \CI(M;\Lambda ^{*+2})\otimes_{\CI(M)}\CI(F;V).
\label{IBG.72}\end{equation}
As in the finite-dimensional case, $R$ commutes with the action of
$\CI(M;\Lambda ^*)$ and so is given by a fibre-wise homomorphism. This is a
2-form on $M$ with values in the (first-order) differential operators on the
fibres of $F.$ From the standard formula 
\begin{equation}
R(v,w)=\nabla_v^V\nabla_w^V-\nabla_w^V\nabla_v^V-\nabla_{[v,w]}^V
\label{IBG.94}\end{equation}
and \eqref{IBG.70} it follows that 
\begin{equation}
R(v,w)=\nabla ^V_{\mu(v,w)}+\beta (\tau (v),\tau (w))
\label{IBG.95}\end{equation}
where $\mu$ is the curvature of the Ehresmann connection and $\beta$ is the
curvature of $\nabla ^V.$

Again as in the finite-dimensional case the connection on $\CI(F;V)$
induces a connection on homomorphisms. In this case for instance the
fibrewise pseudodifferential operators act 
\begin{equation}
\Psi_{\phi}^k(F;V):\CI(M;\Lambda ^*)\otimes_{\CI(M)}\CI(F;V)\longrightarrow \CI(M;\Lambda ^*)\otimes_{\CI(M)}\CI(F;V)
\label{IBG.73}\end{equation}
and then 
\begin{equation}
(\widetilde\nabla E)w=\widetilde\nabla(Ew)-E(\widetilde\nabla w)
\label{IBG.74}\end{equation}
defines a multliplicative connection on $\Psi_{\phi}^k(F;V).$ In terms of a trivialization
of $F$ in which $\CI(F,V)$ is locally trivialized to the bundle $\CI(Z;W)$
over open sets of $M$ this connection is 
\begin{equation}
\nabla =d_M+[\gamma ,\cdot]
\label{IBG.76}\end{equation}
in terms of \eqref{IBG.77}. 

Then, as a fibrewise differential operator, $R\in \CI(M;\Lambda
^2)\otimes_{\CI(M)}\Psi_{\phi}^2(F;V)$ and with respect to the extended
connection the Bianci identity takes the form 
\begin{equation}
\widetilde \nabla R=0.
\label{IBG.75}\end{equation}

The connection $\nabla$ restricts to a connection on the range of the
reference bundle $\Pi_-$ in \eqref{IBG.65} and similarly on the bundle
$\CI(F;V\otimes \bbC^2);$ we will continue to denote either connection by
$\nabla.$ Since it is a fibrewise operator $\Pi$ extends naturally to act on
the bundles $\CI(M;\Lambda ^*)\otimes_{\CI(M)}\CI(F;V).$ Then $\nabla$
induces a connection on the range of $\Pi$ by
\begin{equation}
\Pi w=w,\ \Mie w\in\Ran(\Pi) \Longrightarrow
\nabla^{\Pi}w=\Pi(\nabla w)\in\CI(M;\Lambda ^*)\otimes_{\CI(M)}\Ran(\Pi).
\label{IBG.78}\end{equation}
In a local trivialization of $F,$ which will not in general trivialize
$\Pi,$ this connection differs from the connection on the range of $\Pi_-$
by a fibrewise smoothing operator
\begin{equation}
d_M+\gamma ^{\Pi},\ \gamma ^{\Pi}=\gamma ^-+\mu,%
\gamma ^-=\begin{pmatrix}0&0\\0&\gamma 
\end{pmatrix},\ \mu\in\CI(M;\Lambda ^1)\otimes_{\CI(M)}\Psi^{-\infty}_{\phi}(F;V\otimes\bbC^2).
\label{IBG.79}\end{equation}
Similarly the difference $R^{\Pi}-R^-$ of $R^{\Pi}=(\widetilde\nabla^{\Pi})^2$ and the
curvature $R^-$ of the reference operator $\Pi_-$ is a 2-form on $M$ with values in the
fibrewise smoothing operators.

\begin{proposition}\label{IBG.80} In the model \eqref{IBG.66} for $K^0(M)$
  the Chern character is given by 
\begin{equation}
\Ch(\Pi)=\sum\limits_{k}\frac{1}{(2\pi
  i)^kk!}\Tr\left((R^{\Pi})^k-(R^-)^k\right),\ [\Ch(\Pi)]\in H^{\ev}_{\dR}(M).
\label{IBG.81}\end{equation}
\end{proposition}

\begin{proof} Since $R^{\Pi}-R^-$ is smoothing, so is $(R^{\Pi})^k-(R^-)^k$
  and hence \eqref{IBG.81} defines a form in $\CI(M;\Lambda ^{\ev}).$ That
  this form is closed follows essentially as in the finite-dimensional
  case. Thus, in a local trivialization of $F$ and separately for each
  term,
\begin{equation}
d_M\Tr\left((R^{\Pi})^k-(R^-)^k\right)=\Tr\left(d_M(R^{\Pi})^k-d_M(R^-)^k\right).
\label{IBG.82}\end{equation}
From the trace identity it follows that 
\begin{equation}
\Tr\left([\gamma ^-,(R^{\Pi})^k-(R^-)^k)]\right)=0.
\label{IBG.83}\end{equation}
Then, since $\mu$ is a smoothing operator 
\begin{equation}
\Tr\left([\mu ,(R^{\Pi})^k]\right)=0.
\label{IBG.84}\end{equation}
Combining these shows that 
\begin{equation}
d_M\Tr\left((R^{\Pi})^k-(R^-)^k\right)=\Tr\left(\tilde\nabla^{\Pi}(R^{\Pi})^k-\tilde\nabla(R^-)^k\right)=0
\label{IBG.85}\end{equation}
from the Bianchi identity.

A similar argument, as in the finite-dimensional case, over $[0,1]_t\times
M,$ shows that the deRham class of $\Ch(\Pi)$ is constant under smooth
homotopy so $[\Ch(\Pi)]\in H^{\ev}(M)$ is well-defined, including
independence of the choices of $\nabla^V$ and the Ehresmann connection on $F.$

That this is indeed the Chern character follows from the reduction
\eqref{IBG.68}.
\end{proof}

\section{Odd Chern character}\label{OddCC}
Since it is not quite as familiar as the even Chern character we derive the
odd version by suspension. For a manifold $M $ 
\begin{equation}
K^0(\bbS\times M)=K^0(M)\oplus K^1(M).
\label{IBG.46}\end{equation}
Any odd K-class on $M$ can be constructed, in terms of \eqref{IBG.46}, from
a vector bundle $V$ over $M$ and an invertible section 
\begin{equation}
g\in {\End}(V).
\label{IBG.47}\end{equation}

Then a bundle $U $ over $\bbS\times M$ can be constructed in terms of its
global space of sections by considering those smooth sections of the pull-back of
$V$ to $[0,2\pi+\epsilon)\times M$  
\begin{equation}
u:[0,2\pi+\epsilon)\times M\longrightarrow V\Mst
u(2\pi+s,m)=g(m)u(s,m),\ s\in[0,\epsilon)
\label{IBG.48}\end{equation}
where $\epsilon >0$ is suitably small, e.g. $\epsilon =1/6.$ This is a
module over $\CI(\bbS\times M)$ and has a local basis so does indeed define
a bundle $U\longrightarrow \bbS\times M.$

If $\nabla^V$ is a connection on $V$ then a connection on $U$ can be
constructed using a cut-off $\chi\in\CIc([0,2\pi+\epsilon]$ with
$\chi(s)=0$ in $s<\epsilon$ and $\chi(s)=1$ in $s>2\epsilon.$  Applied to a
global section as in 
\begin{equation}
\nabla^Uu=ds\cdot \pa_su+(\nabla^V-\chi (s)(\nabla^V g)g^{-1})u
\label{IBG.49}\end{equation}
where $\nabla^Vg$ is the action of the induced connection on $\hom(V).$

Extending the connection as usual to $\tilde\nabla_U$ acting on $\Lambda
^*\otimes U,$ the curvature of this connection is 
\begin{equation}
R_U=\tilde\nabla_U^2=-\chi '(s)ds\wedge (\nabla_Vg)g^{-1}+R_V-\chi(s)\left(
[R_V,g]g^{-1}-(\nabla_Vg)g^{-1}\wedge(\nabla_Vg)g^{-1}\right)
\label{IBG.50}\end{equation}
where the curvature of the connection on $\hom(V)$ is the commutator action
of the curvature of $V.$

The Chern character of $U$ as a form on $\bbS\times M$ is therefore 
\begin{equation}
\sum\limits_{k\ge0}\frac{1}{(2\pi i)^kk!}\tr(R_U^k).
\label{IBG.51}\end{equation}
Pushing this forward to $M$ gives the odd Chern character of the class
defined by $g$ and $V:$
\begin{multline}
\Ch_{\odd}(g,V)=\sum\limits_{k\ge1}\int_{0}^{2\pi}\frac1{(2\pi i)^k(k-1)!}
\tr(-\chi '(s)ds\wedge (\nabla_Vg)g^{-1}\\
\left(R_V-\chi(s)(
[R_V,g]g^{-1}-(\nabla_Vg)g^{-1}\wedge(\nabla_Vg)g^{-1})\right)^{k-1}.
\label{IBG.52}\end{multline}
Here we have extracted the coefficient of $ds.$ 
Expanding out the second factor to a part with and a part without $\chi$ gives
multiples of 
\begin{equation}
\int_{0}^{2\pi}\chi(s)\chi^p(s)=\frac1{p+1}=\int_0^1s^p
\label{IBG.53}\end{equation}
so the integral is independent of the choice of $\chi$ and can be written 
\begin{multline}
\Ch_{\odd}(g,V)=\sum\limits_{k\ge1}\frac1{(2\pi i)^k(k-1)!}\int_{0}^{1}ds
\tr(-(g^{-1}\nabla^Vg)\\
\left((1-s)g^{-1}R_Vg-sR_V-(g^{-1}\nabla^Vg)\wedge (g^{-1}\nabla^Vg))\right)^{k-1}.
\label{IBG.54}\end{multline}

By construction this form is closed and the multliplicativity of the even
Chern character implies that it is multiplicative under  
\begin{equation}
\begin{gathered}
K^0(M)\times K^1(M)\ni([U],[(g,V)))\longmapsto ([(g\otimes\Id_U,V\otimes
    U)))\in K^1(M),\\
\Ch_{\odd}((g\otimes\Id_U,V\otimes
    U))=\Ch(U)\wedge\Ch_{\odd}(g,V)).
\end{gathered}
\label{IBG.55}\end{equation}

We pass from this finite-dimensional formula to the Chern character for
the classifying bundle  
\begin{equation}
G^{-\infty}_{\phi}(F;L)\longrightarrow Y
\label{IBG.56}\end{equation}
for any smooth fibre bundle $F\longrightarrow Y$ and vector bundle $L$ over
$F.$

Just as \eqref{IBG.54} gives a representative deRham class for any
connection on $V,$ we choose an Ehresmann connection on $F$ and a connection
on $L$ over $F.$ Then any section of \eqref{IBG.56} is smoothly homotopic
to a finite rank perturbation of the identity onto a smooth
subbundle. Applying \eqref{IBG.54} and taking the limit as the rank goes to
infinity shows that \eqref{IBG.54} extends to give essentially the same
formula
\begin{multline}
\Ch_{\odd}(g)=\sum\limits_{k\ge1}\frac1{(2\pi i)^k(k-1)!}\int_{0}^{1}ds
\tr(-(g^{-1}\nabla g)\\
\left((1-s)g^{-1}Rg-sR-(g^{-1}\nabla g)\wedge (g^{-1}\nabla g))\right)^{k-1}\\
g\in\CI(Y;G^{-\infty}_{\phi}(F;L)
\label{IBG.57}\end{multline}
where now the curvature and connection are those induced on the bundle
$\Psi^{-\infty}_{\phi}(F;L)$ from the connection on $\CI(F;L)$ by the
Ehresmann connection on $F$ and a choice of connection on $L.$
This connection on $\Psi^{-\infty}_{\phi}(F;L)$ is multiplicative. 

\section{Twisted Chern character for small bundle gerbes}
 
Now we pass to the twisted case for a small gerbe and replace
$\Psi^{-\infty}_{\phi}(F;L)$ by the bundle of algebras over $M$ which as a
space (altough not as a bundle of algebras) is
\begin{equation}
\Psi^{-\infty}_{\phi,J}(F)=\Psi^{-\infty}_{\phi}(F)\otimes_{\CI(F^{[2]})}\CI(F^{[2]};J).
\label{IBG.58}\end{equation}
Using an Ehresmann connection on $F,$ and hence on $F^{[2]},$ we may
use \eqref{IBG.58} to induce a multiplicative connection on the twisted
bundle of smoothing algebras. Namely, take a simplicial connection on $J$
over $F^{[2]},$ so consistent with the simplicial trivialization 
\begin{equation}
\pi_{12}^*J\otimes\pi_{23}^*J\simeq \pi_{13}^*J\text{ over }F^{[3]}.
\label{IBG.59}\end{equation}
The consistency of these connections over the multiplicative action of
$\CI(F)$ gives a connection via the tensor product \eqref{IBG.58}.

\begin{proposition}\label{IBG.60} The connection, $\nabla^J$ on
  $\Psi^{-\infty}_{\phi,J}(F)$ induced by an Ehresmann connection on $F,$
  and a simplicial connection on $J$ is multiplicative and has curvature
  \emph{with a global representation,} determined by a choice of a B-field
  for $J,$ as a twisted fibre pseudodifferential operator
\begin{equation}
[R_J,\cdot],\ R_J\in\CI(M;\Lambda ^2)\otimes_{\CI(M)}\Psi^{1}_{\phi,J}(F);
\label{IBG.61}\end{equation}
the Bianchi identity for which becomes 
\begin{equation}
\tilde\nabla^J R_J=T\in\CI(M;\Lambda ^3),\ dT=0
\label{IBG.86}\end{equation}
with $T/2\pi i$ a deRham representative of the Dixmier-Douady class of the gerbe.
\end{proposition}
\noindent Note that the operator $R_J$ is actually the sum of a first-order
twisted differential operator and a twisted smoothing operator.

The curvature of a connection on a bundle such as
$\Psi^{-\infty}_{\phi,J}(F)$ is, in principle a 2-form with values in the
bundle of homomorphisms. Since the connection is mutltiplicative and the
algebra is simple it is natural to expect that the curvature is of the form
\eqref{IBG.61}. However, such a representation is not unique since adding a
2-form on $M,$ acting as a multiple of the identity on the algebra leaves
it unchanged. Thus $R_J$ amounts to a simplicial lift of the curvature.

\begin{proof} In a local trivialization of $F$ and simplicial reduction of
  $J,$ the derivation of \eqref{IBG.95} carries over to the twisted
  case. The first term is the fibre differential operator given by the
  evaluation of the Ehresmann curvature on on connection on $S.$ Acting on
  the left or right this is a well-defined first order differential
  operator in $\Psi^{1}_{\phi,J}(F).$ The second term in \eqref{IBG.95} is
  the curvature of $J$ and here the lift to a commutator depends on the
  simplicial trivialization of $J.$ To ensure a globally defined lift a
  B-field, a 2-form $B$ on $F$ which is a simplicial lift of the curvature
  of $J,$ can be chosen. The resulting splitting of the second term in
  \eqref{IBG.95} as a commutator then differs from local version given by
  the connection on $S\otimes S'$ by the pull-back a \v Cech $2$-form on
  the base.

The Bianchi identity \eqref{IBG.75} holds for the lift of the curvature to
a commutator as given by local simplicial trivialization but the global
version given by the B-field differs from it by the differential of the \v
Cech $2$-form. This however is a global closed $3$-form on $M$ which gives \eqref{IBG.86}.
\end{proof}

In this twisted setting the even twisted K-theory of $M$ can be
represented by idempotents $\Pi$  
\begin{equation}
\Pi=\Pi_-+q,\ \Pi^2=\Pi,\ q\in\Psi^{-\infty}_{\phi,J}(F;\bbC^2),\ \Pi_-=%
\begin{pmatrix}0&0\\0&\Id
\end{pmatrix}
\label{IBG.87}\end{equation}
up to smooth homotopy. The connection on $\Psi^{-\infty}_{\phi,J}(F)$
induces a connection on the bundle of `Toeplitz' algebras 
\begin{equation}
\Pi\Psi^{-\infty}_{\phi,J}(F;\bbC^2)\Pi.
\label{IBG.88}\end{equation}

Following the discussion above of the untwisted case the curvature of this
connection again has a representation as in \eqref{IBG.61}
\begin{equation}
[R^{\Pi}_J,\cdot],\ R^{\pi}_J\in\CI(M;\Lambda ^2)\otimes_{\CI(M)}\Psi^{2}_{\phi,J}(F)
\label{IBG.89}\end{equation}
with the Bianchi identity \eqref{IBG.86} holding for the same form $T$ on $M.$
\begin{theorem}\label{IBG.90} The even twisted Chern character is defined by
  the direct analogue of \eqref{IBG.81} 
\begin{equation}
\Ch_J(\Pi)=\sum\limits_{k}\frac{1}{(2\pi
  i)^kk!}\Tr\left((R^{\Pi}_J)^k-(R^-_J)^k\right),\ [\Ch(\Pi)]\in H^{\ev}_{\dR}(M)
\label{IBG.91}\end{equation}
which gives an element of twisted cohomology $H_J^{ev}M)$
\begin{equation}
(d_M-T/2\pi i)\Ch_J(\Pi)=0.
\label{IBG.92}\end{equation}
\end{theorem}

\begin{proof} The argument proceeds as in the untwisted case except at the
  last step where the modified Bianchi identity \eqref{IBG.86} means that 
\begin{equation}
\tilde\nabla_J^{\Pi}((R_J^{\Pi})^k)-\tilde\nabla_J(R_J^-)^k)=kT((R_J^{\Pi})^{k-1})-(R_J^-)^{k-1})
\label{IBG.93}\end{equation}
since the form $T$ commutes with $R_J^{\Pi}$ and $R_J^-.$ This gives \eqref{IBG.92}.
\end{proof}

The odd twisted Chern character can be analysed as in \S\ref{OddCC}. The
twisted Chern character in the general case was studied in \cite{BCMMS,MS,
  MS06}.

\section{Index in twisted cohomology}\label{Indcoh}

The twisted index formula in cohomology is a straightforward generalization
of the standard result.

\begin{theorem}\label{IBG.191} For the twisted index map as in Theorem \ref{IBG.25}, {the twisted Chern character}
\begin{equation}
\Ch_J(\ind_{\scl,J}(\alpha))=\phi_*(e^B\Ch(\alpha)\wedge\Td(\phi)),\quad a\in K^*(T^*_{\fib}F)
\label{IBG.192}\end{equation}
takes values in the twisted cohomology of the base  $H_J^{ev}(M).$
Here $\Ch_J: K^*(T^*_{\fib}F) \to H_J^{ev}(M)$ and $B$ is a B-field, which is a global 3-form on $F$
satisfying $F^J = pr_1^*(B) -  pr_2^*(B)$, where $F_J$ is the curvature of  a simplicial connection on $J$
and $pr_j$ is projection to the $j^{th}$ factor. In other words, the following diagram commutes:
\begin{equation}
\begin{gathered}
\xymatrix{
K^0(T^*_{\fib}F) \ar[r]^{\ind_{\scl,J}}\ar[d]^{\Td(F/Y) \wedge e^B\Ch}&K^0(Y, \phi^*\cA_J)\ar[d]^{\Ch_{J}}\\
H^{even}(T^*_{\fib}F, dB)\ar[r]^{e^{B_U}}& H^{even}_J(M)
}
\end{gathered}
\end{equation}

\end{theorem}

\begin{proof} 
As in the standard case it suffices to `follow' the twisted Chern
  character through the embedding proof of the index in K-theory. The
  crucial step of extending the index class from the fibration along the
  normal bundle to the embedding follows as in the standard case. Namely
  the Thom class in K-theory for the normal bundle is an untwisted even
  class which contributes the Todd class leading to \eqref{IBG.192}.
\end{proof}

\section{Projective families of Dirac operators and positive scalar curvature}

Consider the small bundle gerbe consisting of the fibre bundle $\phi:Y\to X$ and primitive line bundle 
$J\to Y^{[2]}$.  There is a  line bundle $K_\tau$ over $\phi^{-1}(U) \subset Y$
for trivializing open subsets $U \subset X$ of $J$, that is 
\begin{equation}
\label{trivializing}
J\big|_{\phi^{-1}(U)\times_U \phi^{-1}(U)} \cong K_\tau \boxtimes K_\tau^*,
\end{equation}
where $K_\tau^*$ denotes the dual line bundle to $K_\tau$. A connection $\nabla^\tau$ on
the line bundle $K_\tau$ determines a simplicial (or bundle gerbe) connection $\nabla^J = \nabla^\tau \boxtimes 1 + 1 \boxtimes  \nabla^{\tau *}$
where $\nabla^{\tau *}$ is the dual connection on $K_\tau^*$ .  If $F^J$ denotes the curvature of $\nabla^J$, then 
$F^J = pr_1^*(B) -  pr_2^*(B)$ for some 2-form $B$ on $Y$, and $pr_j$ is projection to the $j^{th}$ factor.
Then $(\nabla^J, B, dB=\phi^*(H)$ is said to be a connection on the bundle gerbe $(J, Y)$, see \cite{MS2}.
This was previously briefly discussed in Proposition \ref{IBG.60}.

 If the fibres of $Y$ are
even-dimensional and consistently oriented and with a consistent $\SpinC$ structure, 
that is the vertical tangent bundle $T(Y/X)$ is a $\SpinC$vector bundle,
let $\Cl(Y/X)$ denote the complex
bundle of Clifford algebras associated to some family of fibre metrics and
let $\bbS$ be the $\bbZ_2$-graded hermitian Clifford Spin module over $Y$ with
unitary Clifford connection $\nabla^\bbS.$

This data determines a projective family of Dirac operators $\eth_\bbS$
acting fibrewise on the sections of $\bbS.$ We can further twist
$\eth_\bbS$ by the connection $\nabla^\tau$ of 
the line bundle $K_\tau$ over $\phi^{-1}(U) \subset Y$
for trivializing open subsets $U \subset X$ of $J$, as in \eqref{trivializing}.
In this way, we get a family of twisted Dirac operators $\eth_{\bbS\otimes K_\tau}$
on the local fibre bundle $\phi:\phi^{-1}(U) \to U$, and
\begin{equation}
\label{Bochner}
\eth_{\bbS\otimes K_\tau}^2 = ({\nabla^{\bbS\otimes K_\tau}})^*  \nabla^{\bbS\otimes K_\tau} + \kappa + 1/2c(F_L) + c(F_\tau)
\end{equation}
is the Bochner-Weitzenbock formula for this family. Here $\kappa$ denotes the scalar curvature along the fibres
and $F_\tau$ is the curvature of a hermitian connection on $K_\tau$, $F_L$ is the curvature of the $\SpinC$ line bundle.
 The sum $1/2 c(F_L) + c(F_\tau)$ is a skew symmetric endomorphism, 
using the fibred Riemannian metric, with parametrised eigenvalues $\{\pm i\lambda_k\}$. Then we have 
the following vanishing theorem,

\begin{lemma}{\label{lemma Dirac}}
Suppose that $\kappa >  4 \sup_U \sum |\lambda_k|$. Then 
$$
{\rm index}(\eth_{\bbS\otimes K_\tau}) =0 \in K^0(U) \stackrel{\Lambda_U}{\cong} K^0(U, \cA_U).
$$
for any choice of section $\tau$.
Taking the Chern character, we see that 
$$
0=\Ch({\rm index}(\eth_{\bbS\otimes K_\tau}))
= \int_{\phi^{-1}(U)/U} {\widehat A}(Y/X) e^{1/2\, c_1(L)}e^{c_1(K_\tau)}
$$
Equivalently, taking the twisted Chern character,
$$
0=\Ch_{B_U}(\Lambda_U({\rm index}(\eth_{\bbS\otimes K_\tau})))
= e^{B_U} \int_{\phi^{-1}(U)/U} {\widehat A}(Y/X) e^{1/2\, c_1(L)}e^{c_1(K_\tau)}
$$
Here $\tau:U\longrightarrow \phi^{-1}(U),$ is a local section and $B_U = \tau^*(B)$.

\end{lemma}
\begin{proof}
Under the assumption of the Lemma, and in the notation above,
we see that $\eth_{\bbS\otimes K_\tau}^2 >0$, therefore we conclude
the first part.

Another choice of section
$\tau':U\longrightarrow \phi^{-1}(U),$ determines another line bundle
$K_{\tau'}$ over $\phi^{-1}(U)\subset Y,$ satisfying
\begin{equation}
K_{\tau} = K_{\tau'} \otimes \phi^*(L_{\tau, \tau'}),
\end{equation}
where $L_{\tau, \tau'} = (\tau, \tau')^*J$.
Note that the curvature of $ \phi^*(L_{\tau, \tau'})$ is horizontal, so is trivial along the fibres of 
$\phi:\phi^{-1}(U) \to U$, therefore it does not contribute 
to the Bochner-Weitzenbock formula. Finally, clearly the following diagram commutes.

\begin{equation}
\begin{gathered}
\xymatrix{
K^0(U) \ar[r]^{\Lambda_U}\ar[d]^{\Ch}&K^0(U, \cA_U)\ar[d]^{\Ch_{B_U}}\\
H^{even}(U)\ar[r]^{e^{B_U}}& H^{even}(U, dB_U)
}
\end{gathered}
\end{equation}

This completes the proof of the Lemma.
\end{proof}

Assuming that the space $X$ is compact, let $U_1, \ldots U_N$ be a finite trivializing open cover of $\phi:Y\to X$, and choosing
finite local sections $\tau_j:U_j\longrightarrow \phi^{-1}(U_j),$ for $j=1, \ldots N$, we see that

\begin{theorem}
In the notation above, suppose that the fibrewise scalar curvature $\kappa$ satisfies
$$\kappa >  4 \max_{j=1}^N   \sup_{U_j} \sum |\lambda_k|,$$ then the index of the projective 
family of Dirac operators vanishes,
$$
\ind_{\scl,J}(\alpha) = \ind_J (\eth_{\bbS\otimes K}) =0 \in K^0(X, \cA),
$$
where $a\in K^\bullet(T^*_{\fib}Y)$ is the symbol of the projective family of Dirac operators.

Taking the twisted Chern character, and using Theorem \ref{IBG.191} , we see that
$$
0=\Ch_J(\ind_{\scl,J}(\alpha) )
=   \int_{Y/X} {\widehat A}(Y/X) e^B e^{1/2\, c_1(L)+c} 
$$
where $c= c_1(K_{\tau_j})$ on $U_j$, which we will argue below that although $K_{\tau_j}$ is not well defined globally, the integrand is
globally well defined.

\end{theorem}

\begin{proof}
We begin with a discussion of $c$ above. Suppose that $U_i \cap U_j \ne \emptyset$. Then $K_{\tau_i} = K_{\tau_j} \otimes \phi^*(L_{\tau_i, \tau_j})$,
therefore $\phi_* \exp(c_1(K_{\tau_i} ))=\phi_*\exp(c_1( K_{\tau_j}))$, that is $\phi_*(\exp(c))$ is globally defined. Therefore the integrand
above is globally well defined.

\end{proof}

\end{document}